\documentclass{article}

\usepackage{amsmath, amsthm, amsfonts}
\usepackage{amssymb} 
\usepackage{graphicx}
\usepackage{dcolumn}
\usepackage{bm}
\usepackage{amsmath}
\usepackage{amssymb}
\usepackage{amsfonts}
\usepackage[hang,raggedright]{subfigure}
\usepackage[lined,ruled]{algorithm2e}
\usepackage{color}
\usepackage{url}
\usepackage{graphicx}
\usepackage{booktabs}
\usepackage{bm,times}
\usepackage{stmaryrd}

\newcommand{\figref}[1]{F{\textsc{ig}}.~\ref{#1}}

\newtheorem{theorem}{Theorem}
\newtheorem{lemma}{Lemma}

\newcommand{\wt}[1]{\widetilde{#1}}
\newcommand{\wh}[1]{\widehat{#1}}

\newcommand{\pars}[1]{\left( {#1} \right)}

\newcommand{\dd}[1]{\mathrm{d}{#1}}

\newcommand{\sint}[4]{\int_{#1}^{#2}{#3}\,\dd{#4}}

\newcommand{\bs}{\boldsymbol}

\numberwithin{equation}{section}

\definecolor{White}{rgb}{0.99,1.0,1.0}

\definecolor{Purple}{rgb}{0.59, 0.44, 0.84}
\newcommand{\tpurp}[1]{\textcolor{Purple}{#1}}

\title{Efficient Distribution Estimation and Uncertainty Quantification for Elliptic Problems on Domains with Stochastic Boundaries}

\author{Jehanzeb H Chaudhry\thanks{Department of Mathematics and Statistics, University of New Mexico, Albuquerque, NM, 87131 ({\tt jehanzeb@unm.edu}). J. Chaudhry’s work is supported in part by the Department of Energy (DESC0009324)
and by Sandia National Laboratories: Laboratory Directed Research
and Development (LDRD) Funding under Academic Alliance Program
FY2016.} 
\and 
Nathanial Burch\thanks{Gonzaga University, Spokane, Washington, 99258 ({\tt burchn@gonzaga.edu}). N.~Burch's research was partially supported by the National Science Foundation through
	the Statistical and Applied Mathematical Sciences Institute, grant DMS-0635449.}
        \and Donald Estep\thanks{Department of Statistics, Colorado State University, Fort Collins, Colorado, 80523 ({\tt estep@stat.colostate.edu}). D.~Estep's work is
supported in part by the Defense Threat Reduction Agency
(HDTRA1-09-1-0036), Department of Energy (DE-FG02-04ER25620,
DE-FG02-05ER25699, DE-FC02-07ER54909, DE-SC0001724, DE-SC0005304,
INL00120133, DE0000000SC9279), Idaho National Laboratory (00069249,
00115474), Lawrence Livermore National Laboratory (B573139, B584647,
B590495),  National Science Foundation (DMS-0107832, DMS-0715135,
DGE-0221595003, MSPA-CSE-0434354, ECCS-0700559, DMS-1065046,
DMS-1016268, DMS-FRG-1065046, DMS-1228206), National Institutes of
Health (\#R01GM096192).}}

\begin{document}

\maketitle

\begin{abstract}
	We study the problem of uncertainty quantification for the numerical solution of elliptic partial differential equation boundary value problems posed on domains with stochastically varying boundaries. We also use the uncertainty quantification results to tackle the efficient solution of such problems. 
	We introduce simple transformations that map a family of domains with stochastic boundaries to a fixed reference domain. We exploit the transformations to carry out a prior and a posteriori error analyses and to derive an efficient Monte Carlo sampling procedure.
\end{abstract}




\section{Introduction}

In this paper, we study uncertainty quantification and efficient solution of boundary value problems for elliptic partial differential equations (PDEs) posed on domains with stochastically perturbed boundaries. The problem of stochastic boundaries occurs for a variety of reasons, e.g. from physical stresses, manufacturing deficiencies, and uncertainty in measurements of a fixed geometry. Specific applications are found in transport in tubes with rough
boundaries \cite{tartakovsky2006stochastic}, aerodynamic studies in the design of
wind turbines \cite{canuto2009numerical}, heat diffusion across irregular and fractal-like surfaces
\cite{blyth2003heat,brady1993diffusive},
structural analysis studies \cite{nouy2008extended}, acoustic scattering on rough surfaces \cite{oba2010global,xiu2007efficient},
seismology and oil reservoir management \cite{babuska2004galerkin},
various civil and nuclear engineering studies \cite{babuska2005worst},
chemical transport in rough domains \cite{broyda2010probability},
and electromechanical studies for nanostructures \cite{arnst2009probabilistic}.

This paper focuses on two key issues that arise in such problems: 
\begin{itemize} 
\item Since the geometric properties of
the domain has a strong effect on solution behavior
and smoothness,  significant variation in
solution behavior for different realizations of the domain is to be expected. Correspondingly, significant variation in the error arising from
discretization and sampling is also to be expected; 
\item Each realization of a domain nominally requires construction of a new discretization mesh, at a significant computational cost. Hence, solving such problems using a Monte Carlo approach is computationally intensive.
\end{itemize}

\tpurp{We deal with these issues using two ideas. First, motivated  by the technique of isoparametric finite elements \cite{ciarlet1978finite}, we describe a family of  simple, locally determined, well-behaved transformations that map
a given elliptic problem posed on a family of stochastic domains $\{\Omega(\bs \theta)\}_{\theta\in \Theta}$ to an elliptic problem with stochastic
coefficients posed on a fixed reference domain $\Omega$. We then construct finite element approximations for the solution of sample elliptic problems on the reference domain and use these to formulate a Monte Carlo method  to compute the sample cumulative distribution function for a specified Quantity of Interest (QoI).  We carry out a full {a priori} analysis of the finite element method. 
Second, we carry out an uncertainty quantification by deriving a posteriori error estimates both for a QoI computed from a numerical solution that take into account all sources of deterministic and stochastic errors and for the approximate cumulative distribution function computed from the QoI. The estimate is sufficiently detailed that we can efficiently balance computational work, e.g. mesh resolution versus sample numbers, to achieve a desired accuracy. This provides a way to tackle computational efficiency by describing an efficient adaptive strategy that leads to a mesh that produces acceptable accuracy for
all realizations of the problem.}

\tpurp{The analysis of elliptic problems posed on stochastic domains has received much less attention than elliptic problems with stochastic coefficients. An approach that has received substantial attention is based on postulating global transformations between a reference domain and the random domains that satisfy certain regularity conditions and the use of  Karhunen-Loeve expansions, stochastic Galerkin methods, stochastic collocation, etc., to compute numerical approximations. Some of the earlier references are \cite{tartakovsky2006stochastic,xiu2007efficient,xiu2007numerical, arnst2009probabilistic,oba2010global}, which considers elliptic problems posed on a domain whose boundary is parameterized by a stochastic process.   Some of the key technical issues were analyzed in subsequent work, e.g., \cite{harbrecht2010output, CNT16,HPS16,HSSCS15}.}

\tpurp{The mathematical analysis of  the method considered in this paper and the methods studied in \cite{tartakovsky2006stochastic,xiu2007efficient,xiu2007numerical, arnst2009probabilistic,oba2010global,harbrecht2010output, CNT16,HPS16, HSSCS15} deal with similar technical challenges since all the methods employ transformations between random domains and a reference domain. However, the method analyzed in this paper has several fundamental differences to these other methods.  The great difficulty involved in constructing smooth global transformations between two  domains places constraints on the formulation of the random domain problem. The method analyzed in this paper avoids construction of global transformations between the random domain realizations and a fixed reference domain, instead employing simple, easily computed transformations confined to a neighborhood of the random boundary. This makes the method well-suited for the class of problems where data describing the boundary of each realization of the random domain is given, e.g., determined through physical measurement. Some of the analysis is  focused on dealing with the technical issues that arise from the use of a localized transformation.  With the goal of accurately computing a sample cumulative distribution function as opposed to a couple of statistics, we employ Monte Carlo sampling, and we consider the issue of increasing the efficiency of this approach. Finally, in addition to {\em a priori} convergence analysis, we present and implement {\em a posterior} error analysis as well as an adaptive method based on the estimates.}

Of course, transforming a domain of a given elliptic problem to another domain is a classic analytic approach,  e.g., Schwarz-Christoffel
transformations \cite{driscoll2002schwarz,trefethen1979numerical} provide a global and smooth map. However, conformal maps involve significant complications in practice, e.g., the resulting maps may introduce complications such as singularities at boundaries of a doman and they are difficult and expensive to compute. This motivates the very simple transformations studied in this manuscript.

The rest of the paper is organized as follows.
In Section~\ref{sec:prob_formulation}, we present the problem formulation and modeling assumptions.
We construct a piecewise-smooth transformation to a deterministic domain $\Omega$,
provide details for the finite element method, formulate the adjoint problem and an error estimate in Section~\ref{sec:problem_transformed}.
The estimate is extended to analyze Lions domain decomposition on a transformed domain in Section~\ref{sec:lions_domain_decomp}.
A posteriori estimates of the various sources of errors for each realization and for the empirical distribution
function are obtained in Section~\ref{sec_posteriori}.
Section~\ref{sec:AMR} discusses the construction of a finite element mesh suitable for all realizations of $\bs\theta$.
Numerical experiments are performed and presented throughout the paper.

\section{Problem formulation}
\label{sec:prob_formulation}
\subsection{Stochastically perturbed domains}
We describe the domains with stochastic boundaries as random perturbations of a nominal deterministic reference domain $\Omega$. The reference domain $\Omega$ is a convex polygonal domain in $\mathbb{R}^2$ with sides formed by straight edges joining a collection of nodes $\{\hat v_j\}_{j=1}^J$. To define the stochastic perturbations, we let $\{\hat \theta_j\}_{j=1}^J$ denote a collection of random vectors $\hat \theta_j \in \mathbb{R}^2$ such that $\hat \theta_j \sim (\Lambda_j, \mathcal{F}_j, P_j)$, where $(\Lambda_j, \mathcal{F}_j, P_j)$ is a probability space with compact domain $\Lambda_j \subset \mathbb{R}^2$, $\sigma$-algebra $\mathcal{F}_j$, and probability measure $P_j$. We abuse notation to let $\bs\theta = \{\hat \theta_j\}_{j=1}^J$ and let $\bs \theta \in \Theta$ denote the set of admissible perturbation vectors. The stochastic perturbation $\Omega(\bs\theta)$ is the polygonal domain with boundaries defined by straight edges connecting nodes $\{\hat v_j + \hat \theta_j\}_{j=1}^J$, ; see \figref{domains_from_data}.
\begin{figure}[htbp]
	\begin{center} \includegraphics[width=0.99\textwidth]{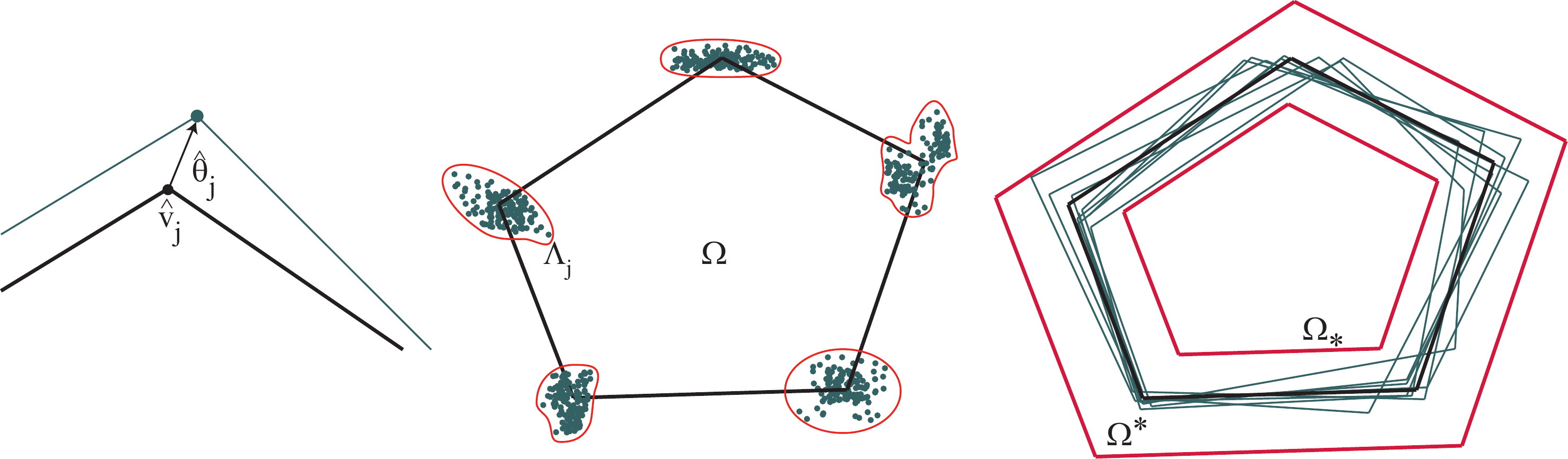}
	\end{center}
\caption{Left: Stochastic perturbation of a boundary node. Middle: A collection of perturbed boundary nodes. Right: A collection of stochastic perturbations of the nominal domain.} \label{domains_from_data}
\end{figure}
We note that $\Omega(\bs\theta) \to \Omega$ as $\| \bs \theta\| = \max \|\hat \theta_j \| \to 0$, where $\| \quad \|$ in the second quantity denotes the Euclidean norm.

To enforce the assumption that the domains $\Omega(\bs\theta)$ share the same gross shape, we let $\Omega^{\ast}$ and  $\Omega_{\ast}$ denote convex polygonal domains with $J$ boundary nodes and nonzero volume obtained by scaling $\Omega$ so that $\Omega_{\ast} \subseteq \Omega \subseteq \Omega^{\ast}$. We assume that for all $\bs\theta$,
\begin{equation*}
	\Omega_{\ast} \subseteq \Omega(\bs\theta) \subseteq \Omega^{\ast}, \quad \forall \bs\theta \in \Theta, \notag
\end{equation*}
see Fig.~\ref{domains_from_data}. As we demonstrate, the solution technique described below can be applied to relatively large perturbations and to nonconvex domains. However, there are well known analytic difficulties associated with elliptic problems on non-convex polygonal domains with ``sharp'' inset angles.  By reducing the volume of $\Omega^{\ast} \setminus \Omega_{\ast}$, we can limit the degree of non-convexity in $\Omega(\bs\theta)$ that may occur. 

We emphasize that the numerical method does \textbf{not} depend on identifying a ``true'' nominal domain $\Omega$. Under these assumptions, given a collection of sample domains $\{\Omega(\bs\theta^n)\}$ corresponding to a collection of samples $\{\bs \theta^n\}$,  we can simply choose one sample domain to use as a reference domain. As we discuss, some choices may yield improved numerical accuracy.

\subsection{The elliptic problem}
In defining the elliptic problem, we wish to avoid situations in which there is a radical change in material properties from one domain to the next. So, we assume that the coefficients and right-hand side of the elliptic equation are defined in the largest domain $\Omega^*$. 
The elliptic problem on $\Omega(\bs\theta)$ is: Find $w$ satisfying
\begin{equation}
	\left\{
	\begin{aligned}
		-\nabla \cdot \left( a(x) \nabla w(x;\bs\theta) \right) &= f(x), & x &\in \Omega(\bs\theta), \\
		w(x;\bs\theta) &= 0, & x &\in \partial\Omega(\bs\theta),
	\end{aligned} \right. \label{model_random_domain}
\end{equation}
where $a$ is a $2\times2$ symmetric positive definite matrix with coefficients that are continuous functions on $\Omega^{\ast}$ such that there is a constant $a_0>0$ with
$w^\top a w \geq a_0 \|w\|^2$ for all $w \in \mathbb{R}^2$ and $x \in \Omega^{\ast}$ and $f$ is continuous on $\Omega^{\ast}$. We assume that both $a$ and $f$ are independent of $\bs\theta$, though extensions to the cases when $a$ and $f$ are also stochastic are straightforward \cite{estep2008nonparametric,estep2009nonparametric}.
Having obtained the solution, we evaluate the quantity of interest (QoI),
\begin{equation}
\label{eq:qoi_orig_def}
Q(w;\bs\theta) = \sint{\Omega(\bs\theta)}{}{w(x;\bs\theta)\psi(x)}{x},
\end{equation}
where $\psi \in L^2(\Omega^{\ast})$.  By standard results $w(x;\bs\theta)$ depends continuously on $\bs\theta$. Hence, $Q(w;\bs\theta)$ is a random variable and the stochastic version of \eqref{model_random_domain} and \eqref{eq:qoi_orig_def} is to compute the probability distribution of $Q(w;\bs\theta)$,
\begin{align}
	P(t) = \text{Pr}(Q(u;\bs\theta) \leq t). \notag
\end{align}
Since we are dealing with a collection of domains, there are some restrictions on possible QoIs. For example, it is inappropriate to choose the value at a point that is not in all the sample domains. Some choices of QoI vary as the domain varies, e.g.,  $\psi = 1$.

\subsection{Basic properties of the elliptic problem}

The variational formulation of \eqref{model_random_domain} reads:
Find $w \in H_0^1(\Omega(\bs\theta))$ such that
\begin{equation} \label{model_random_domain_variational_1}
	B_{\Omega(\bs\theta)}(w,v) = F_{\Omega(\bs\theta)}(v), \quad \forall v \in H_0^1(\Omega(\bs\theta)),
\end{equation}
where the bilinear form
\begin{align}
	B_{\Omega(\bs\theta)}(w,v) = \sint{\Omega(\bs\theta)}{}{a(x) \nabla w \cdot \nabla v}{x} \notag
\end{align}
is bounded and coercive for all $\bs\theta \in \Theta$, i.e.,
\begin{subequations}
\label{eq:coer_cont_consts_transformed}
\begin{align}
		|B_{\Omega(\bs\theta)}(w,v)|
			&\leq C_1(\bs\theta) \|w\|_{H_0^1(\Omega(\bs\theta))} \|v\|_{H_0^1(\Omega(\bs\theta))},
					& \forall w,v \in H_0^1(\Omega(\bs\theta)), \notag \\
		B_{\Omega(\bs\theta)}(w,w)
			&\geq C_2(\bs\theta) \|w\|^2_{H_0^1(\Omega(\bs\theta))},
					& \forall w \in H_0^1(\Omega(\bs\theta)), \notag
\end{align}
\end{subequations}
and the linear form
\begin{align}
	F_{\Omega(\bs\theta)}(v) = \sint{\Omega(\bs\theta)}{}{f(x)v(x)}{x} \notag
\end{align}
is bounded, i.e.,
\begin{align}
	|F_{\Omega(\bs\theta)}(v)|
		\leq C_3(\bs\theta) \|v\|_{H_0^1(\Omega(\bs\theta))}. \notag
\end{align}
We assume uniform boundedness and coercivity with respect to $\bs\theta$, i.e., there are constants $C_1,C_2,C_3$ with
\begin{align}
	C_1(\bs\theta) \leq C_{1,max} < \infty
	\quad
	C_2(\bs\theta) \geq C_{2,min} >0, \quad  C_3(\bs\theta) \leq C_{3,max} < \infty . \label{uniform_bounds_1}
\end{align}

The Lax-Milgram lemma (\cite{larsson2008partial}) grants a unique weak solution $w$ to \eqref{model_random_domain_variational_1} for all $\bs\theta \in \Theta$. Moreover, by \eqref{uniform_bounds_1},
\begin{align}
	\sup_{\bs\theta \in \Theta} \| w \|_{H_0^1(\Omega(\bs\theta))} \leq \frac{1}{C_{2,max}} \| f \|_{L^2(\Omega^{\ast})}. \notag
\end{align}
Under the convexity assumption, the standard regularity result holds, i.e.,
\begin{align}
	\| w \|_{H_0^2(\Omega(\bs\theta))} \leq C(\bs\theta) \| f \|_{L^2(\Omega(\bs\theta))}. \notag
\end{align}

\section{Transformed Problem}
\label{sec:problem_transformed}

The problem is transformed by mapping $\Omega(\bs\theta)$ back to the reference domain $\Omega$ using a piecewise affine map. As mentioned, we can choose any sample domain to be the reference domain. To create the map, we partition $\Omega$ into $D$ triangular subdomains $\{\Omega_d\}_{d=1}^D$ such that $\overline{\Omega} = \bigcup_{d=1}^D \overline{ \Omega_d}$. The subdomains in the partition are chosen so that they are non-intersecting and the vertices of subdomains do not intersect interiors of edges of other subdomains or interiors of boundary edges. Under these assumptions, there is a corresponding partition $\{\Omega_d(\bs\theta)\}_{d=1}^D$ of each $\Omega(\bs\theta)$ that is obtained by perturbing only the nodes of the partition $\{\Omega_d\}_{d=1}^D$ that lie on the boundary of $\Omega$ by $\bs \theta$.

\tpurp{There are many choices of such partitions but some choices yield better numerical results than others. In particular, we use the partition to make a domain decomposition formulation of the original elliptic problem and the difficulty in obtaining accurate numerical solutions is affected by the shape of the subdomains. Consequently, the angles in the subdomains and number of subdomains impact the convergence of the domain decomposition iteration and the condition numbers of the resulting linear systems. Likewise, the properties of the transformations affect the accuracy of the finite element approximation. We illustrate some of the issues with an example in \S \ref{sec:num_ex_trans_dom_effect}.}

\subsection{Random samples}
We select $N$ independent realizations $\{ \bs\theta^n \}_{n=1}^N$, which corresponds to a set of sample domains $\{\Omega(\bs\theta^n))\}$.

\subsection{Transformation to a reference domain}

Rather than discretizing the elliptic problem on each sample domain $\Omega(\bs\theta^n)$ directly, we first  apply a piecewise affine map $\varphi : \Omega(\bs\theta^n) \to \Omega$. The map is determined by $D$ invertible affine maps $\varphi_d(\bs\theta^n)=\varphi_d^n : \Omega_d(\bs\theta^n) \to \Omega_d$, $1 \leq d \leq D$.  Let $y \in \Omega_d$ denote the image of $x \in \Omega_d(\bs\theta^n)$ under the map $\varphi_d^n$ and let $\mathbf{J}_d^n$ denote the Jacobian matrix of $\varphi_d^n$. We denote the three vertices of $\Omega_d(\bs\theta^n)$ by $\mathbf{r}_{d,1}^n$, $\mathbf{r}_{d,2}^n$, and $\mathbf{r}_{d,3}^n$ and the three corresponding vertices of $\Omega_d$ by $\mathbf{s}_{d,1}$, $\mathbf{s}_{d,2}$, and $\mathbf{s}_{d,3}$,  see Fig.~\ref{diff_problems_meshed}. The transformation is defined,
\begin{align}
	\varphi_d^n(x) = \mathbf{J}_d^n(x-\mathbf{r}_{d,1}^n) + \mathbf{s}_{d,1}, \label{define_phi}
\end{align}
where
\begin{align}
	\mathbf{J}_d^n =
	\begin{pmatrix}
		\mathbf{s}_{d,2}-\mathbf{s}_{d,1} & -(\mathbf{s}_{d,3}-\mathbf{s}_{d,1})
	\end{pmatrix}
	\begin{pmatrix}
		\mathbf{r}_{d,2}^n-\mathbf{r}_{d,1}^n & -(\mathbf{r}_{d,3}^n-\mathbf{r}_{d,1}^n)
	\end{pmatrix}^{-1} = \mathbf{S}_d(\mathbf{R}_d^n)^{-1}. \notag
\end{align}
Both  $\mathbf{S}_d$ and $\mathbf{R}_d^n$ are invertible since $\Omega_d(\bs\theta^n)$ and $\Omega_d$ are
non-degenerate triangles.
\begin{figure}[htbp]
	\begin{center}
\includegraphics[width=0.99\textwidth]{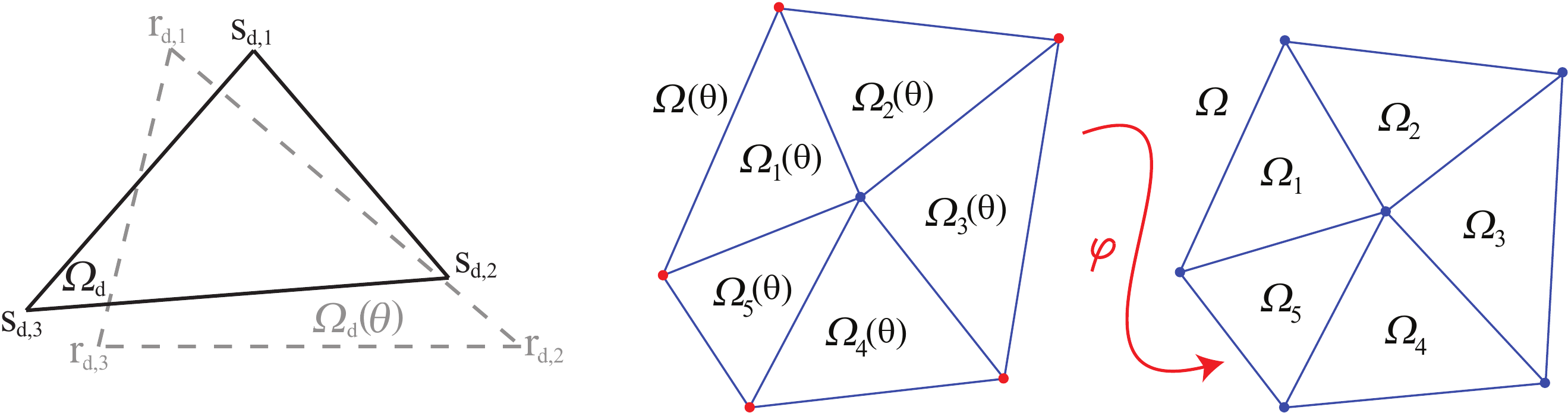}
		\caption{Left: $\Omega_d$ and $\Omega_d(\bs\theta)$. Right: The transformation $\varphi: \Omega(\bs\theta) \to \Omega$. We have dropped the superscript indicating sample.} \label{diff_problems_meshed}
	\end{center}
\end{figure}

We constrain the partitions $\{\Omega_d(\bs\theta^n)\}_{d=1}^D$ so that for constants $M_{\ast}, M^{\ast}$,
\begin{align}
\label{eq:subdomain_pert_cond}
	0 < M_{\ast}
		\leq \min_{n,d} |\det(\mathbf{J}_d^n)|
		\leq \max_{n,d} |\det(\mathbf{J}_d^n)|
		\leq M^{\ast} < \infty,
\end{align}
uniformly for all samples $\{\bs\theta^n\}$.
We note that \eqref{eq:subdomain_pert_cond} implies uniform bounds on $\| \mathbf{J}_d^n \|$ and $\| (\mathbf{J}_d^n)^{-1} \|$.

\subsection{Transformation of the problem}
\label{sec:transformed_problem_poisson}
We next reformulate the elliptic problem by applying the transformation $\varphi$.

Let $\mathbf{n}_d$ denote the outward pointing normal vector to the boundary of $\Omega_d(\bs\theta^n)$, $a|_{\Omega_d} \equiv a_d$ and $f|_{\Omega_d} \equiv f_d$. Then, for $1\leq d \leq D$, compute  $w_d$ on $\Omega_d(\bs\theta^n)$ solving,
\begin{equation} \label{random_problem_partitioned_1}
	\left\{
	\begin{aligned}
		-\nabla \cdot \pars{ a_d(x) \nabla w_d^n(x;\bs\theta^n) } &= f_d(x), & x &\in \Omega_d(\bs\theta^n), \\
		w_d^n(x;\bs\theta^n) &= 0, & x &\in \partial\Omega(\bs\theta^n) \cap \partial\Omega_d(\bs\theta^n), \\
		w_d^n(x;\bs\theta^n) &= w_{\wt{d}}^n(x;\bs\theta^n), & x &\in \partial\Omega_d(\bs\theta^n) \cap \partial\Omega_{\wt{d}}(\bs\theta^n), & \forall \wt{d} &\in d', \\
		\mathbf{n}_d \cdot (a_d(x) \nabla w_d^n(x;\bs\theta^n))
			&= -\mathbf{n}_{\wt{d}} \cdot (a_d(x) \nabla w_{\wt{d}}^n(x;\bs\theta^n)), & x &\in \partial\Omega_d(\bs\theta^n) \cap \partial\Omega_{\wt{d}}(\bs\theta^n), & \forall \wt{d} &\in d',
	\end{aligned}
	\right.
\end{equation}
where $d'$ is the set of $\{ 1,\hdots,D \} \backslash \{d\}$ such that $\Omega_d(\bs\theta^n)$ and $\Omega_{\wt{d}}(\bs\theta^n)$ share a boundary. The last two lines in \eqref{random_problem_partitioned_1} are interface conditions guaranteeing continuity of the solution and normal flux across the boundaries.
The analysis of the existence, uniqueness, and regularity of the solution of \eqref{random_problem_partitioned_1} is discussed in
\cite{babuska1970finite,bjorstad1986iterative,bramble1986iterative,lions1988schwarz,lions1990schwarz}. We have $w |_{\Omega_d} \equiv w_d$.

Using \eqref{define_phi} in \eqref{random_problem_partitioned_1}, we obtain the transformed problem on $\Omega$,
\begin{equation} \label{mapped_problem_2d_0}
	\left\{
	\begin{aligned}
		-\nabla \cdot \pars{ \mathbf{A}_d^n(y;\bs\theta^n) \nabla u_d^n(y;\bs\theta^n) } &= F_d^n(y;\bs\theta^n), & y &\in \Omega_d, \\
		u_d^n(y;\bs\theta^n) &= 0, & y &\in \partial\Omega \cap \partial\Omega_d, \\
		u_d^n(y;\bs\theta^n) &= u_{\wt{d}}^n(y;\bs\theta^n), & y &\in \partial\Omega_d \cap \partial\Omega_{\wt{d}}, & \forall \wt{d} &\in d', \\
		\mathbf{n}_d \cdot (\mathbf{A}_d(y;\bs\theta^n) \nabla u_d^n(y;\bs\theta^n))
				&= -\mathbf{n}_{\wt{d}} \cdot (\mathbf{A}_{\wt{d}}(y;\bs\theta^n) \nabla u_{\wt{d}}^n(y;\bs\theta^n)),
							& y &\in \partial\Omega_d \cap \partial\Omega_{\wt{d}}, & \forall \wt{d} &\in d',
	\end{aligned}
	\right.
\end{equation}
where $u_d(y;\bs\theta^n) = w_d((\varphi_d^n)^{-1}(y);\bs\theta^n)$,
\begin{align}
	\mathbf{A}_d^n(y;\bs\theta^n) = \lvert \mathrm{det} \mathbf{J}_d^n \lvert ^{-1}\; \mathbf{J}_d^n \;  a_d((\varphi_d^n)^{-1}(y)) \; (\mathbf{J}_d^n)^\top,
	 \text{  and,  } 
	F_d^n(y;\bs\theta^n) = \lvert \mathrm{det} \mathbf{J}_d^n \lvert ^{-1} f_d((\varphi_d^n)^{-1}(y)). \notag
\end{align}
Equivalently, we can write \eqref{mapped_problem_2d_0} in compact form as,
\begin{equation}
\label{mapped_problem_2d_0_compact}
\left \{
\begin{aligned}
-\nabla \cdot \pars{ \mathbf{A}^n(y;\bs\theta^n) \nabla u^n(y;\bs\theta^n) } &= F^n(y;\bs\theta^n), & y &\in \Omega,\\
u^n(y;\bs\theta^n) &= 0, & y &\in \partial\Omega \cap \partial\Omega,
\end{aligned}
\right.
\end{equation}
where $u^n(y;\bs\theta^n) = u_d^n(y;\bs\theta^n)$, $\mathbf{A}^n(y;\bs\theta^n) = \mathbf{A}_d^n(y;\bs\theta^n)$ and $F^n(y;\bs\theta^n) = F_d^n(y;\bs\theta^n)$ for $y \in \Omega_d$.
We note that \eqref{define_phi} implies that $\mathbf{A}_d^n$ and $\mathbf{F}_d^n$ can be computed without constructing $\Omega(\bs\theta^n)$ explicitly.

The QoI in terms of the transformed variables becomes,
\begin{equation}
\label{eq:qoi_trans}
Q(u^n; \bs\theta) = \int_\Omega u^n(y;\bs \theta) \tilde{\psi}(y) \, \mathrm{d}y,
\end{equation}
where $\tilde{\psi}(y)  = \psi( (\varphi_d^n)^{-1}(y)) \,   \lvert \mathrm{det} \mathbf{J}_d^n \lvert ^{-1}$ for $y \in \Omega_d$.

\subsubsection{Weak form of the transformed problem}

For the variational formulation of \eqref{mapped_problem_2d_0_compact}, we let $V= H_0^1(\Omega)$ and  seek $u \in V$, we have,
\begin{equation}
\label{eq:weak_form_trans}
	\begin{aligned}
		\sint{\Omega}{}{\mathbf{A}^n(y;\bs\theta^n)\nabla u^n \cdot \nabla v}{y}
			= \sint{\Omega}{}{F^n(y;\bs\theta^n) v}{y},
	\end{aligned}
\end{equation}
for all $v \in V$.
The matrix $\mathbf{A}^n$ is symmetric, positive definite.

We now show that the bilinear form $ \sint{\Omega_d}{}{\mathbf{A}^n(y;\bs\theta^n) \nabla u_d^n \cdot \nabla v_d}{y} $ is continuous and coercive. Let $\kappa_d := \mathrm{diam} (\Omega_d)$, $\rho_d:= \sup\{ \mathrm{diam}(S) \lvert S \subset \Omega_d$,  $\kappa_d(\bs\theta^n) := \mathrm{diam} (\Omega_d(\bs\theta))$, $\rho_d(\bs\theta^n):= \sup\{ \mathrm{diam}(S) \lvert S \subset \Omega_d(\bs\theta^n)$. We utilize the following properties of the Jacobians $\mathbf{J}_d^n$~\cite{ciarlet1978finite},
\begin{equation}
\label{eq:norm_j_and_j_inv}
\| \mathbf{J}_d^n \| \leq \frac{\kappa_d}{\rho_d(\bs \theta^n)}, \qquad \| (\mathbf{J}_d^n)^{-1} \| \leq \frac{\kappa_d(\bs \theta^n)}{\rho_d}.
\end{equation}
Further,
\begin{equation}
\label{eq:abs_det}
| \det \mathbf{J}_d^n | = \frac{\mathrm{meas}(\Omega)}{\mathrm{meas}(\Omega(\bs \theta^n))},
\end{equation}
and
\begin{equation}
\label{eq:meas_dom}
\pi \rho_d(\bs \theta^n)^2 \leq \mathrm{meas}(\Omega_d(\bs \theta^n)) \leq \pi \kappa_d(\bs \theta^n)^2,\qquad \pi \rho_d^2 \leq \mathrm{meas}(\Omega_d) \leq \pi \kappa_d^2,
\end{equation}
where $\mathrm{meas}$ denotes the Lebesgue measure on $\mathbb{R}^2$.

We have from \eqref{eq:abs_det} and \eqref{eq:meas_dom},
\begin{equation}
\label{eq:det_bounds}
\frac{\rho_d^2}{\kappa(\bs \theta^n)^2}| \leq \lvert \det J_d \lvert \leq \frac{\kappa_d^2}{\rho_d(\bs \theta^n)^2}.
\end{equation}

\begin{lemma}
\label{lem:trans_cont_coer}
The bilinear form
\begin{align}
	B_{\Omega}(u^n,v) = \sint{\Omega}{}{\mathbf{A}^n(y;\bs\theta^n) \nabla u^n \cdot \nabla v}{y} \notag
\end{align}
is bounded and coercive, i.e.,
\begin{subequations}
\begin{align}
		B_{\Omega}(u^n,v)
			&\leq C_{1,n} \|u^n\|_{H_0^1(\Omega)} \|v\|_{H_0^1(\Omega)},
					& \forall u^n,v \in H_0^1(\Omega), \label{eq:cont_trans_form}\\
		B_{\Omega}(u^n,u^n)
			&\geq C_{2,n} \|u^n\|^2_{H_0^1(\Omega)},
					& \forall u \in H_0^1(\Omega), \label{eq:coer_trans_form}
\end{align}
\end{subequations}
where the constants $C_{1,n}$ and $C_{2,n}$ depend on the transformations $\varphi_d(\bs \theta^n)$.
\end{lemma}
\begin{proof}
Let $\lambda_{n,d,max}$ and $\lambda_{n,d,min}$ denotes the largest and smallest eigenvalues of $\mathbf{J}_d^n$ and $a_{max}$ and $a_{min}$ be the ones for the symmetric positive definite matrix $a$. Then,
\begin{equation}
\label{eq:bilin_trans_cont_a}
B_{\Omega}(u^n,v) \leq  \max_d  \left( \lvert \det J_d^{-1} \lvert  \, \lambda_{n,d,max}^2  \right) \,a_{max} \, \| u^n \|_{H_0^1(\Omega)} \| v\|_{H_0^1(\Omega)}.
\end{equation}
By \eqref{eq:norm_j_and_j_inv} 
\begin{equation}
\label{eq:eig_det_J_ub}
\lambda_{n,d,max}^2 = \|\mathbf{J}_d^n (\mathbf{J}_d^n)^\top \| \leq \| \mathbf{J}_d^n \|^2 \leq \frac{\kappa_d^2}{\rho_d(\bs \theta^n)^2}.
\end{equation}
Combining \eqref{eq:det_bounds}, \eqref{eq:bilin_trans_cont_a} and \eqref{eq:eig_det_J_ub} we have,
\begin{equation*}
B_{\Omega}(u^n,v)
			\leq a_{max}  \max_d \left( \frac{\kappa_d^2}{\rho_d^2} \frac{\kappa_d(\bs \theta)^2}{\rho_d(\bs \theta^n)^2} \right)\|u^n\|_{H_0^1(\Omega)} \|v\|_{H_0^1(\Omega)},
\end{equation*}
which proves \eqref{eq:cont_trans_form} with 
\begin{equation}
C_{1,n} = a_{max}  \max_d \left( \frac{\kappa_d \, \kappa_d(\bs \theta^n)}{\rho_d \, \rho_d(\bs \theta^n)} \right)^2.
\end{equation}

Now we prove coercivity of the bilinear form. We have,
\begin{equation}
\label{eq:bilin_trans_coer_a}
B_{\Omega}(u^n,u^n) \geq \gamma \min_d \left( \lvert \det J_d^{-1} \lvert  \, \lambda_{n,d,min}^2  \right) \,a_{min} \, \| u^n \|^2_{H_0^1(\Omega)} 
\end{equation}
\tpurp{where $\gamma$ is the constant arising from Poincar\'e's inequality.} 
By \eqref{eq:norm_j_and_j_inv},
\begin{equation*}
\lambda_{n,d,min}^2 = \frac{1}{\| (\mathbf{J}_d^n (\mathbf{J}_d^n)^\top)^{-1}\|} 
 \geq \frac{1}{( \|\mathbf{J}_d^n)^{-1} \|^2} \geq \frac{\rho_d^2}{\kappa_d(\bs \theta^n)^2}.
\end{equation*}
Combining this with \eqref{eq:det_bounds} and \eqref{eq:bilin_trans_coer_a},
\begin{equation*}
B_{\Omega}(u^n,u^n)
			\geq a_{min}  \gamma \min_d \left( \frac{\rho_d^2}{\kappa_d^2} \frac{\rho_d(\bs \theta)^2}{\kappa_d(\bs \theta^n)^2} \right)\|u^n\|^2_{H_0^1(\Omega)},
\end{equation*}
which proves \eqref{eq:coer_trans_form} with
\begin{equation}
C_{2,n} = \gamma a_{min}  \min_d \left( \frac{\rho_d \, \rho_d(\bs \theta^n)}{\kappa_d \, \kappa_d(\bs \theta^n)} \right)^2.
\end{equation}
\end{proof}

A similar argument shows that the linear form $\ \int_{\Omega_d} F^n(y;\bs\theta) v_d(y)\, dy $ 	 is  bounded.

\subsection{Finite element discretization} \label{subsec_fem}

We next discretize each transformed problem \eqref{mapped_problem_2d_0_compact} using a standard  finite element method corresponding to a triangulation of $\Omega$. We let $\mathcal{T}_h$ denote a triangulation of $\Omega$ that is a refinement of the partition $\{\Omega_d\}$, that is $\mathcal{T}_h$ is a collection of non-overlapping triangular elements $\{K_m\}_{m=1}^M$ that is constructed by starting with $\{\Omega_d\}$ and refining into smaller triangles such that no node of one $K_m$ intersects and interior edge of another $K_{m^\prime}$ and $\Omega = \cup_m K_m$. We let $h_K$ denote the length of the maximum side of element $K$, define the mesh function $h(x) = h_K$ for $x\in K$, and set $h = \max h_K$ for all $K$ in $\mathcal{T}_h$. We let $\alpha_K$ denote the maximum of the interior angles in element $K$, and assume there is a constant $\alpha < \pi$ such that $\alpha_K < \alpha$ for all $K \in \mathcal{T}_h$. The maximum angle condition insures that the finite element solution corresponding to the mesh converges at the expected rate.
We let $V_{h}(\Omega)$ denote the space of continuous piecewise linear functions on $\mathcal{T}_h$. We note that the restriction of $V_{h}(\Omega)$  is a subset of $V$. The finite element discretization reads: Compute $U^n \in V_{h}(\Omega)$ such $\forall v \in V_{h}(\Omega)$,
\begin{equation} \label{weak_form_of_problem}
	\begin{aligned}
		 \sint{\Omega}{}{\mathbf{A}_d^n(y;\bs\theta^n)\nabla U^n \cdot \nabla v}{y}
			= \sint{\Omega}{}{F_d^n(y;\bs\theta^n) v}{y}.
	\end{aligned}
\end{equation}

\subsection{Convergence properties}

The restriction on the partitions $\Omega_d$ in \eqref{eq:subdomain_pert_cond}, along with the maximum angle condition implies that the finite element approximation converges at a first order rate in the energy norm and a second order rate in the $L^2(\Omega)$ norm uniformly with respect to $\bs\theta^n$. However, the choice of partitioning has a significant effect on the magnitude of the error, even if the mesh satisfies the maximum angle condition which can be observed by consideration of the constants in the convergence for the  $H^1(\Omega)$ norm. A standard result from finite element analysis is that the $H^1(\Omega)$ norm is bounded by,
\begin{equation}
\|u^n - U^n \|_{H_0^1(\Omega)} \leq \sqrt{\frac{C_{1,n}}{C_{2,n}}} \min_{v \in V_h(\Omega)} \|u^n - v \|_{H_0^1(\Omega)},
\end{equation}
where $C_{1,n}$ and $C_{2,n}$ are the continuity and coercivity constants for the transformed problem, see Lemma~\ref{lem:trans_cont_coer}. The factor $\min_{v \in V_h(\Omega)} \|u^n - v \|_{H_0^1(\Omega)}$ depends on the order of the polynomials used in constructing $ V_h(\Omega)$ as well as the properties of the triangulation $\mathcal{T}_h$. This term is well understood and here we focus on the effect on the error due to the transformation.  Substituting the values of the constants $C_{1,n}$ and $C_{2,n}$ from the proof of Lemma~\ref{lem:trans_cont_coer} we have,
\begin{multline}
\label{eq:up_bnd_h1_err_trans}
\|u^n - U^n \|_{H_0^1(\Omega)} \leq \sqrt{  \frac{a_{max}}{\gamma a_{min}}} \max_d\left( \frac{\kappa_d \, \kappa_d(\bs \theta^n)}{\rho_d \, \rho_d(\bs \theta^n)} \right)
 \cdot \left [ \min_d \left( \frac{\rho_d \, \rho_d(\bs \theta^n)}{\kappa_d \, \kappa_d(\bs \theta^n)} \right) \right ]^{-1} \min_{v \in V_h(\Omega)} \|u^n - v \|_{H_0^1(\Omega)} \\
  = \sqrt{ \frac{a_{max}}{\gamma  a_{min}}}  \max_d\left( \frac{\kappa_d \, \kappa_d(\bs \theta^n)}{\rho_d \, \rho_d(\bs \theta^n)} \right)^2\min_{v \in V_h(\Omega)} \|u^n - v \|_{H_0^1(\Omega)}  
 \end{multline}

\tpurp{If we assume that we have the same domain $\Omega$ with partitions $\Omega_d$ for the transformed domain, then the bound \eqref{eq:up_bnd_h1_err_trans} indicates that the error is adversely affected by the shape of the random domain $\Omega(\bs \theta)$. In particular, sliver-like partitions, that is partitions which have  a large $\kappa_d(\bs \theta)/\rho_d(\bs \theta)$ ratio, reflect a boundary geometry that can have large impact on the accuracy of the finite element approximation.}

\subsection{Construction of the transformation and a numerical example}\label{sec:num_ex_trans_dom_effect}

\tpurp{As indicated, there is a great deal of flexibility in the construction of the transformations to a reference domain. We choose transformations that are the identity except for a neighborhood of the boundary.  We construct the subdomains  $\Omega_d$ using a uniform partition of $\Omega$. The affine maps $\varphi_d^n$ are defined to be the identity for the subdomains  which do not intersect with the boundary of the domain, that is, the subdomains $\Omega_d(\bs \theta^n)$ and $\Omega_d$ coincide in this case. This effectively localizes the domain transformation to a neighborhood of the boundary. The maps $\varphi_d^n$ for the subdomains which intersect the boundary are specified by the the formula \eqref{define_phi}. Thus as the points on the boundary change, 
the subdomains $\Omega_d(\bs \theta^n)$ which intersect with the boundary and  the maps $\varphi_d^n$ change while the corresponding subdomains $\Omega_d$  remain fixed. We illustrate in Figure~\ref{fig:rand_doms_num_exp_effect_trans}.}

\begin{figure}[htbp]
	\begin{center}
		\subfigure[$\Omega(\bs \theta^1)$]{\includegraphics[width=0.3\textwidth]{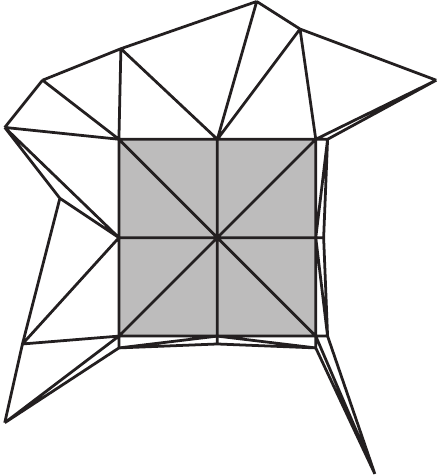}}
		\hspace{0.2in}
		\subfigure[$\Omega(\bs \theta^2)$]{\includegraphics[width=0.3\textwidth]{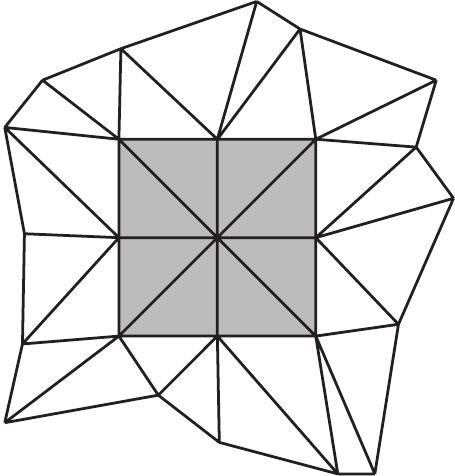}}
		\hspace{0.2in}
		\subfigure[$\Omega$]{\includegraphics[width=0.3\textwidth]{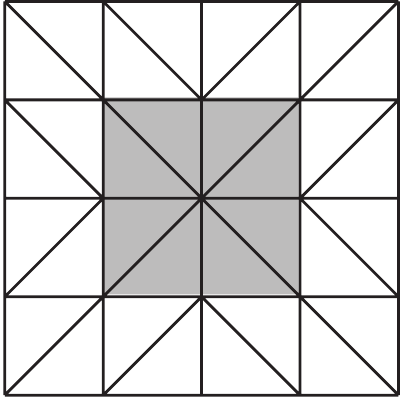}}
		\caption{The random domains $\Omega(\bs \theta^1)$  and $\Omega(\bs \theta^2)$  and the reference domain $\Omega$ to which they are mapped. The transformations are the identity on the cells in the grey shaded region.}
		\label{fig:rand_doms_num_exp_effect_trans}
	\end{center}
\end{figure}

\tpurp{The random domain $\Omega(\bs \theta^1)$ is sliver like and $\displaystyle \max_d \left( \frac{\kappa_d(\bs \theta^1)}{\rho_d(\bs \theta^1)} \right)=32.3$.  The random domain $\Omega(\bs \theta^2)$ is more regular and $\displaystyle \max_d \left( \frac{\kappa_d(\bs \theta^2)}{\rho_d(\bs \theta^2)} \right)=4.8$.}

\tpurp{Reflecting boundary geometry, a transformation may have a strong effect on the difficulty in obtaining an accurate approximation. We illustrate by solving an elliptic problem with $a=1$, $f(x,y) = 200x(1-x)+200y(1-y)$ on the two random domains $\Omega(\bs \theta^1)$ and $\Omega(\bs \theta^2)$ shown in Figure~\ref{fig:rand_doms_num_exp_effect_trans}. For the QoI, we choose $\psi = \chi 10 x y$ in \eqref{eq:qoi_orig_def} where $\chi$ is the characteristic function of $[0.50, 0.75] \times [0.50,0.75]$.}  

\tpurp{We compute numerical approximations for both $\bs \theta^1$ and $\bs \theta^2$ using the transformed problem formulation on $\Omega$ in \S \ref{sec:problem_transformed}. The mesh for $\Omega$ has 681 vertices and was generated  using Gmsh~\cite{GR09}. The numerical solutions are computed using the standard space of piecewise linear continuous functions.}

\begin{figure}[htbp]
	\begin{center}
		\subfigure[$U^1$]{\includegraphics[width=0.4\textwidth]{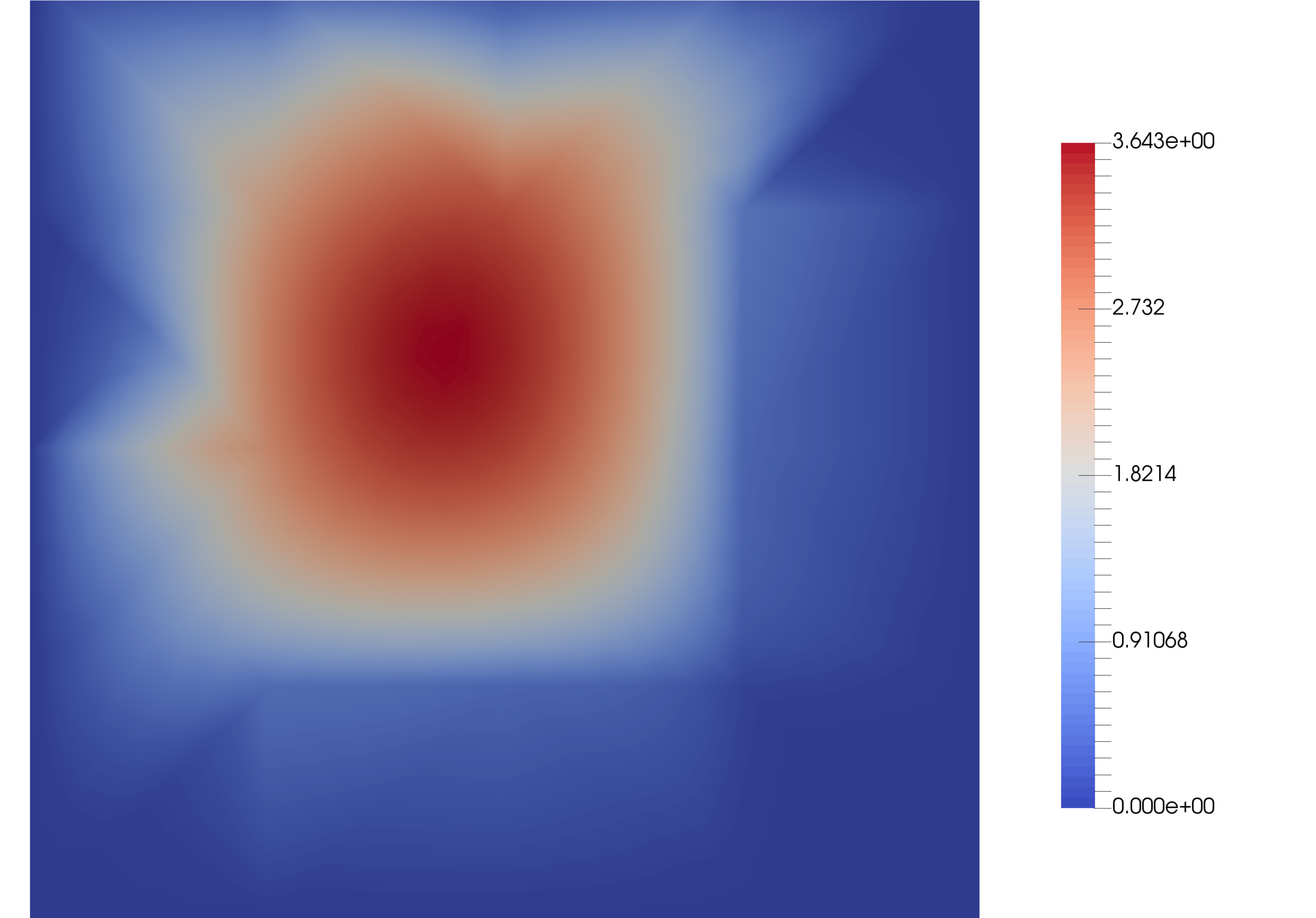}}
		\hspace{0.6in}
		\subfigure[$U^2$]{\includegraphics[width=0.4\textwidth]{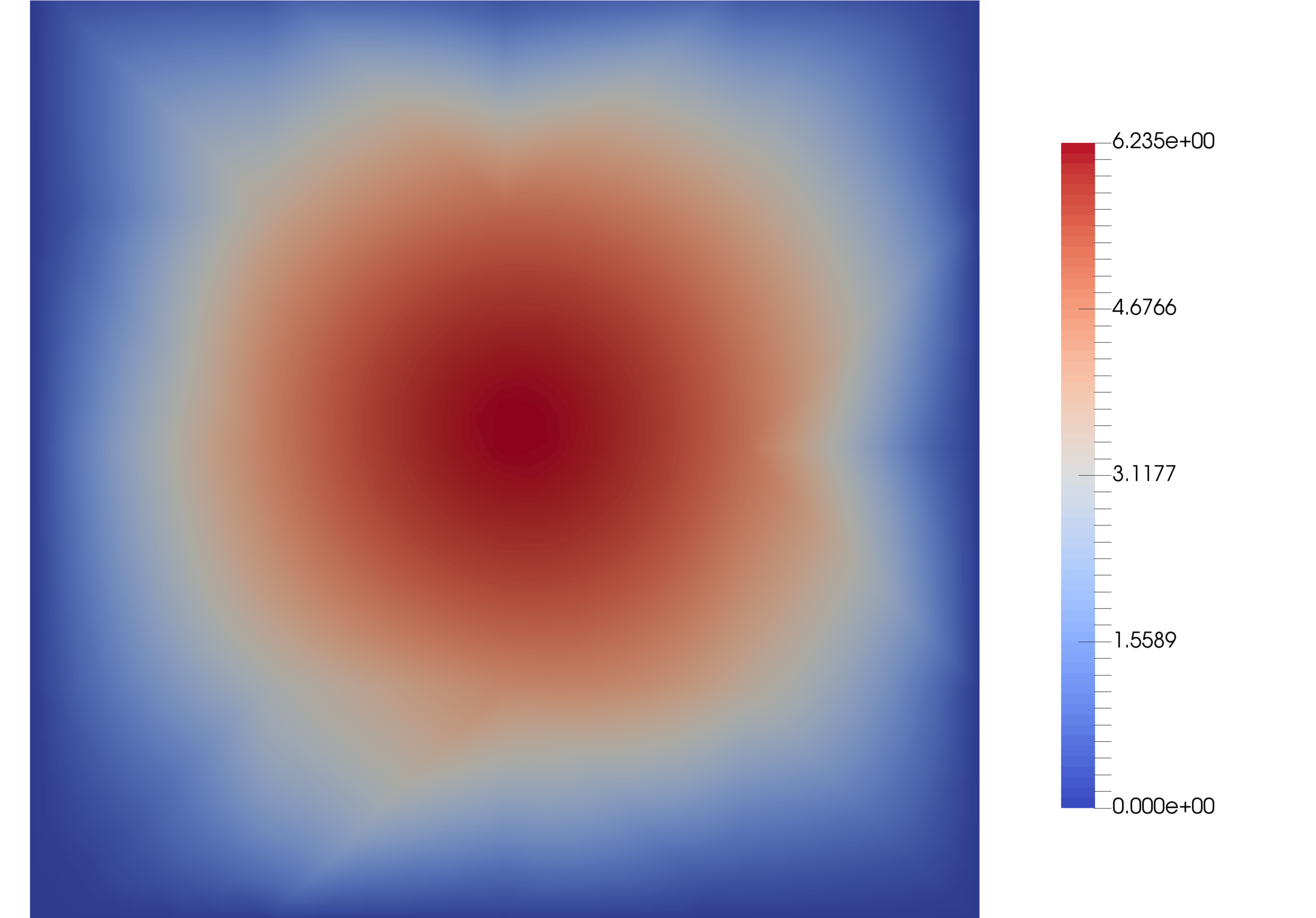}}		 
		\caption{ Solutions of the two transformed problems on $\Omega$ obtained by mapping the random domains $\Omega(\bs \theta^1)$ and $\Omega(\bs \theta^2)$ shown in Figure~\ref{fig:rand_doms_num_exp_effect_trans}.}
		\label{fig:rand_doms_num_exp_solns}
	\end{center}
\end{figure}

\tpurp{The plots of the transformed numerical solutions, $U^1$ and $U^2$, corresponding to $\bs \theta^1$ and $\bs \theta^2$ respectively are shown in Figure~\ref{fig:rand_doms_num_exp_solns}. The effects of sliver like subdomains present in $\Omega(\bs \theta^1)$ are apparent in the plot of $U^1$ while the plot of $U^2$ exhibits a much smoother solution.
The error in the QoI  for $U^1$  is $8.6E-3$ whereas the error in the QoI for $U^2$ is $ 7.4 E-3$. These errors are approximated by using  more accurate reference values of the QoI computed using piecewise quadratic continuous elements on a refined mesh of the untransformed domain. The error in the QoI for $U^1$ is greater than the error in the QoI for $U^2$.
The large error for $U^1$ is expected since the ratio $\displaystyle \max_d \left( \frac{\kappa_d(\bs \theta^1)}{\rho_d(\bs \theta^1)} \right)$ is quite large and hence the error is adversely affected in light of \eqref{eq:up_bnd_h1_err_trans}.}

\subsection{A posteriori analysis of the transformed problem}
We  derive an accurate, computable a posteriori error estimate for the QoI computed from a numerical solution. The estimate employs computable residuals and a generalized Green's function satisfying an adjoint equation. The Green's function quantifies the effects of stability in the accumulation, cancellation and propagation of the errors in the QoI. 
The strong form of the adjoint problem corresponding to the transformed problem \eqref{mapped_problem_2d_0_compact} is,
\begin{equation}
\label{eq:adjoint_strong_form_trans}
\left \{
\begin{aligned}
-\nabla \cdot \pars{ \mathbf{A}^n(y;\bs\theta^n) \nabla \eta^n(y;\bs\theta^n) } &= F^n(y;\bs\theta^n), & y &\in \Omega,\\
\eta^n(y;\bs\theta^n) &= \tilde{\psi}(y), & y &\in \partial\Omega \cap \partial\Omega,
\end{aligned}
\right.
\end{equation}
where $\tilde{\psi}(y)$ is given as in \eqref{eq:qoi_trans}.

\begin{theorem}[Error Representation for the Transformed Problem] 

Let $e^n = u^n - U^n $. Then we have the error representation,
\begin{equation}
\label{eq:err_est_trans}
Q(e^n; \bs \theta^n) = \int_\Omega F^n \eta^n \, \mathrm{d}y - \int_\Omega \mathbf{A}^n \nabla U^n \cdot \nabla \eta^n \,\mathrm{d}y.
\end{equation}
\end{theorem}

\begin{proof}
The proof is standard, e.g. see  Section 8.1 in \cite{eehj_actanum_95}.
\end{proof}

\subsubsection{Numerical Example for error estimate of the transformed problem}
\tpurp{
In the earlier numerical example in \S \ref{sec:num_ex_trans_dom_effect}, the errors in the QoI corresponing to the solutions $U^1$ and $U^2$ were   $Q(e^1; \bs \theta^1) = 8.6E-3$ and $Q(e^2; \bs \theta^2) = 8.6E-3$ respectively. We now estimate the error using the error estimate obtained from \eqref{eq:err_est_trans}. The adjoint solution is computed using the standard space of piecewise quadratic continuous functions. The error estimate obtained for $Q(e^1; \bs \theta^1)$ is  $8.06E-3$ while for $Q(e^1; \bs \theta^1)$ is $7.2E-3$. These have effectivity ratios of $0.93$ and $0.97$ and hence are quite accurate.}

\section{Lions Non-overlapping Domain Decomposition}
\label{sec:lions_domain_decomp}
Equation \eqref{weak_form_of_problem} yields a large coupled system of discrete equations for the finite element approximation.  Following \cite{estep2008nonparametric,estep2009nonparametric}, it is natural to solve the discrete equations using a non-overlapping domain decomposition iteration. In this case, the a posteriori error analysis must be extended to include the effects of the iterative solution of the accuracy of the approximation. For notational simplicity, we drop the superscript $n$ indicating the sample number. In particular, \eqref{eq:weak_form_trans} is rewritten as,
\begin{equation}
\label{eq:weak_form_trans_no_n}
	\begin{aligned}
		\sint{\Omega}{}{\mathbf{A}(y;\bs\theta)\nabla u \cdot \nabla v}{y}
			= \sint{\Omega}{}{F(y;\bs\theta) v_d}{y},
	\end{aligned}
\end{equation}

The non-overlapping domain decomposition solution is formed by employing the Lions domain decomposition algorithm~\cite{lions1990schwarz}. The continuous Lions problem is to find an iterative solution $u_d^{(i)}$,  where the superscript ${(i)}$ refers to the approximation at iteration $i$.. That is, given an set of initial guesses $\lbrace u_d^{(0)}, d = 1, \ldots , D\rbrace$ for the $D$ subdomains, we solve for $i = 1 , 2, \ldots $,
\begin{equation}
\left\{
\begin{aligned}
-\nabla  \cdot \mathbf{A}_d \nabla u_d^{(i)} &= F_d, \quad &y \in \Omega_d,\\
\lambda u_d^{(i)} + \mathbf{n}_d \cdot \mathbf{A}_d \nabla u_d^{(i)} &= \lambda u_{\tilde{d}}^{i-1} - \mathbf{n}_{\tilde{d}} \cdot \mathbf{A}_{\tilde{d}}(y)
 \nabla u_{\tilde{d}}^{i-1}, \quad &y \in \partial \Omega_d \cap \partial \Omega_{\tilde{d}},\\
 u_d &= 0, \quad &y \in \partial \Omega_d \cap \partial \Omega,
\end{aligned}
\right.
\end{equation}
where $\mathbf{n}_d  = -\mathbf{n}_{\tilde{d}}$ is the unit normal and  $\lambda$ is some constant. Lions proved that as $i \rightarrow \infty$, $u_d^{(i)}(y) \rightarrow u(y)$ ~\cite{lions1990schwarz}. 

The discrete analog of Lions domain decomposition involves finding a numerical solution for a finite number of iterations $I$. We let the restriction of $\mathcal{T}_h^d$ on $\Omega_d$ be $\mathcal{T}_h^d$ and define $V^h(\Omega_d) \subset H^1(\Omega_d)$ as the set of continuous piecewise linear functions on $\mathcal{T}_h^d$. That is, given an set of initial guesses $\lbrace U_d^{(0)}, d = 1, \ldots , D\rbrace$ for the $D$ subdomains and fixed $\lambda \geq 0$, we solve for $U_d^{(i)} \in V_h(\Omega_d)$,  $i = 1 , 2, \ldots, I$,
\begin{equation}
\begin{aligned}
 (\mathbf{A}_d \nabla U_d^{(i)}, \nabla v)_d &+ \sum_{\tilde{d} \in d^\prime}  \lambda \langle  U_d^{(i)},v\rangle_{d \cap \tilde{d}} \\
  &= (F_d,v)_d + \sum_{\tilde{d} \in d^\prime}  \left(    \lambda \langle  U_{\tilde{d}}^{(i-1)},v\rangle_{d \cap \tilde{d}}  - \langle \mathbf{n}_{\tilde{d}} \cdot \mathbf{A}_{\tilde{d}} \nabla U^{(i-1)}_{\tilde{d}} , v \rangle_{d \cap \tilde{d}}\right)
\end{aligned}
\end{equation}
for all $v \in V^h(\Omega_d)$  where,
\begin{equation*}
(f,g)_d = \int_{\Omega_d} f(y) g(y) \, \mathrm{d}y \qquad \text{ and } \qquad \langle f,g \rangle_{d \cap \tilde{d}} = \int_{\partial \Omega_d \cap \partial \Omega_{\tilde{d}}} f(y) g(y) \, \mathrm{d} y.
\end{equation*}

\subsection{A posteriori analysis of Lions Domain Decomposition}
We denote the numerical solution after $i$ iterations as $U^{(i)}$ so that $U^{(i)}\lvert_{\Omega_d} = U^{(i)}_d$. Then the computed value of the QoI at iteration  is obtained from \eqref{eq:qoi_trans} as,
\begin{equation}
Q(U^{(i)}) = \int_\Omega U^{(i)} \tilde{\psi} \, \mathrm{d}y = \sum_{d=1}^D \int_{\Omega_d} U_d^{(i)} \tilde{\psi} \, \mathrm{d}y = \sum_{d=1}^D (U_d^{(i)} , \tilde{\psi})_d.
\end{equation}

\begin{theorem}

Let $s^{(i)} = u - U^{(i)}$. Then the error in the QoI at iteration $i$ of Lion's domain decomposition algorithm is represented as,
\begin{equation}
Q(u - U^{(i)}) = DE^{(i)} + IE^{(i)} + CE^{(i)},
\end{equation}
where $DE^{(i)}$, $IE^{(i)}$, $CE^{(i)}$ represent the discretization, iteration and transformation contributions to the total error and are given as,
\begin{align*}
DE^{(i)}  &=   \sum_{d=1}^D \bigg[ (F,\eta)_d +
-( \mathbf{A}_d \nabla U^{(i)}_d, \nabla \eta )_d  + \sum_{\tilde{d}\in d^\prime}  \left(   \langle -\lambda U_d^{(i)} +  \lambda U_{\tilde{d}}^{(i-1)}- \mathbf{n}_{\tilde{d}} \cdot \mathbf{A}_{\tilde{d}} \nabla U^{(i-1)}_{\tilde{d}} , \eta \rangle_{d \cap \tilde{d}}   \right)
  \bigg],\\\
IE^{(i)}  &= \sum_{d=1}^D \sum_{\tilde{d}\in d^\prime} \left(   \langle  \lambda  U_d^{(i)}-\lambda  U_{{d}}^{(i-1)} , \eta \rangle_{d \cap \tilde{d}} +  \langle  \mathbf{n}_d \cdot \nabla U_d^{(i-1)}- \mathbf{n}_d \cdot \nabla U_{{d}}^{(i)} , \eta \rangle_{d \cap \tilde{d}}\right), \\
CE^{(i)}&=    \frac{1}{2} \sum_{d=1}^D \sum_{\tilde{d}\in d^\prime}  \left(
 \langle \mathbf{n}_{{d}} \cdot  \mathbf{A}_{{d}} \nabla U^{(i)}_{{d}}  -  \mathbf{n}_{{d}} \cdot \mathbf{A}_{\tilde{d}} \nabla U^{(i)}_{\tilde{d}}, \eta \rangle_{d \cap \tilde{d}} + 
 \langle \mathbf{n}_d \cdot  \mathbf{A}_d \nabla \eta,  U_d^{(i)} - U_{\tilde{d}}^{(i)} \rangle_{d \cap \tilde{d}} \right)
\end{align*}

\end{theorem}

\begin{proof}
Multiplying \eqref{eq:adjoint_strong_form_trans} by $e^{(i)}_d$ and integrating by parts on $\Omega_d$ yields,
\begin{equation}
\label{eq:lions_qoi_proof_a}
(e^{(i)}, \tilde{\psi})_d = ( \mathbf{A}_d \nabla e^{(i)}_d, \nabla \eta )_d - \sum_{\tilde{d} \in d^\prime} \langle \mathbf{n}_d \cdot \mathbf{A}_d \nabla \eta , e_d^{(i)}\rangle_{d \cap \tilde{d}}.
\end{equation}
Summing \eqref{eq:lions_qoi_proof_a} over all domains, and noting that since $u$ is the true solution, we have $\langle \mathbf{n}_d \cdot \mathbf{A}_d \nabla \eta, u \rangle_{d \cap \tilde{d}} = -\langle \mathbf{n}_{\tilde{d}} \cdot \mathbf{A}_d \nabla \eta  , u\rangle_{d \cap \tilde{d}} $,
\begin{equation}
\label{eq:lions_qoi_proof_b}
Q(e^{(i)}) =  \sum_{d=1}^D \bigg[ (F,\eta)_d
-( \mathbf{A}_d \nabla U^{(i)}_d, \nabla \eta )_d + \sum_{\tilde{d} \in d^\prime} \langle \mathbf{n}_d \cdot  \mathbf{A}_d \nabla \eta, U_d^{(i)}\rangle_{d \cap \tilde{d}} \bigg].
\end{equation}
where we also used \eqref{eq:weak_form_trans}. 
Now consider,
\begin{equation}
\label{eq:lions_qoi_proof_c}
 \sum_{d=1}^D \sum_{\tilde{d}\in d^\prime}   \langle -\lambda U_{{d}}^{(i)} + \lambda U_d^{(i-1)}, \eta
 \rangle_{d \cap \tilde{d}} = \sum_{d=1}^D \sum_{\tilde{d}\in d^\prime}    \langle \lambda  U_d^{(i)}- \lambda U_{{d}}^{(i-1)}  , \eta
 \rangle_{d \cap \tilde{d}}  
\end{equation}
Also,
\begin{equation}
\label{eq:lions_qoi_proof_d}
\begin{aligned}
&\sum_{d=1}^D \sum_{\tilde{d}\in d^\prime}  \langle \mathbf{n}_{{d}} \cdot \mathbf{A}_{{d}} \nabla U^{(i-1)}_{{d}} , \eta \rangle_{d \cap {d}} - \langle \mathbf{n}_{{d}} \cdot \mathbf{A}_{{d}} \nabla U^{(i)}_{{d}} , \eta \rangle_{d \cap {d}}, \\
&= \sum_{d=1}^D \sum_{\tilde{d}\in d^\prime}  \langle \mathbf{n}_{\tilde{d}} \cdot \mathbf{A}_{\tilde{d}} \nabla U^{(i-1)}_{\tilde{d}} , \eta \rangle_{d \cap {d}}- \langle \mathbf{n}_{{d}} \cdot \mathbf{A}_{{d}} \nabla U^{(i)}_{{d}} , \eta \rangle_{d \cap {d}}. \\
&= \sum_{d=1}^D \sum_{\tilde{d}\in d^\prime}  \langle \mathbf{n}_{\tilde{d}} \cdot \mathbf{A}_{\tilde{d}} \nabla U^{(i-1)}_{\tilde{d}} , \eta \rangle_{d \cap {d}}- 
\frac{1}{2} \langle \mathbf{n}_{{d}} \cdot \mathbf{A}_{{d}} \nabla U^{(i)}_{{d}}  + \mathbf{n}_{\tilde{d}} \cdot \mathbf{A}_{\tilde{d}} \nabla U^{(i)}_{\tilde{d}}, \eta \rangle_{d \cap {d}},\\
&= \sum_{d=1}^D \sum_{\tilde{d}\in d^\prime}  \langle \mathbf{n}_{\tilde{d}} \cdot \mathbf{A}_{\tilde{d}} \nabla U^{(i-1)}_{\tilde{d}} , \eta \rangle_{d \cap {d}}- 
\frac{1}{2} \langle \mathbf{n}_{{d}} \cdot \mathbf{A}_{{d}} \nabla U^{(i)}_{{d}}  - \mathbf{n}_{{d}} \cdot \mathbf{A}_{\tilde{d}} \nabla U^{(i)}_{\tilde{d}}, \eta \rangle_{d \cap {d}}, 
\end{aligned}
\end{equation}
where we used   $\mathbf{n}_d = -\mathbf{n}_{\tilde{d}}$ in  the last step. Similarly,
\begin{equation}
\label{eq:lions_qoi_proof_e}
\sum_{d=1}^D \sum_{\tilde{d} \in d^\prime} \langle \mathbf{n}_d \cdot  \mathbf{A}_d \nabla \eta, U_d^{(i)}\rangle_{d \cap \tilde{d}} = \frac{1}{2}\sum_{d=1}^D \sum_{\tilde{d} \in d^\prime} \langle \mathbf{n}_d \cdot  \mathbf{A}_d \nabla \eta, U_d^{(i)} - U_{\tilde{d}}^{(i)}\rangle_{d \cap \tilde{d}}
\end{equation}
Combining \eqref{eq:lions_qoi_proof_c}, \eqref{eq:lions_qoi_proof_d}, \eqref{eq:lions_qoi_proof_e} with \eqref{eq:lions_qoi_proof_b} and grouping terms proves the theorem.
\end{proof}

\subsubsection{Numerical Example of Lions Domain Decomposition}
The setup is similar to \ref{sec:num_ex_trans_dom_effect} corresponding to the parameter $\bs \theta^1$.
Lions domain decomposition is carried out on the transformed problem on $\Omega$. The parameter $\lambda = 5.0$ is used. 

The results are shown in Table~\ref{tab:lions_dd_1}.
\begin{table}[!ht]
\centering
\begin{tabular}{|c|c|c|c|c|c|}
\hline 
$i$ & Comp. Err & Eff. Rat. & $DE^{(i)}$ & $IE^{(i)}$ & $CE^{(i)}$\\
\hline 
1 & 0.3642 & 1.002 & 0.001795 & 0.8971 & -0.5347 \\
3 & -0.0678 & 0.9699 & 0.003994 & 0.525 & -0.5968 \\
5 & -0.1984 & 0.9913 & 0.006685 & 0.0838 & -0.2889 \\
7 & -0.08888 & 0.9928 & 0.007465 & -0.1487 & 0.05233 \\
9 & 0.02633 & 1.008 & 0.006977 & -0.138 & 0.1573 \\
11 & 0.07065 & 0.9981 & 0.00638 & -0.03515 & 0.09943 \\
13 & 0.04921 & 0.9983 & 0.0061 & 0.0281 & 0.015 \\
15 & 0.02459 & 1.001 & 0.006136 & 0.02659 & -0.008139 \\
17 & 0.02107 & 1.005 & 0.006222 & 0.001381 & 0.01346 \\
19 & 0.02831 & 1.004 & 0.006238 & -0.008807 & 0.03088 \\
21 & 0.0324 & 1.003 & 0.006215 & -0.00388 & 0.03007 \\
23 & 0.03097 & 1.002 & 0.0062 & 0.002108 & 0.02267 \\
25 & 0.02886 & 1.003 & 0.006205 & 0.002413 & 0.02024 \\
27 & 0.02858 & 1.003 & 0.006211 & 6.454e-05 & 0.0223 \\
29 & 0.02935 & 1.003 & 0.006212 & -0.0009761 & 0.02411 \\
31 & 0.02976 & 1.003 & 0.006209 & -0.0004158 & 0.02396 \\
33 & 0.0296 & 1.003 & 0.006207 & 0.0002579 & 0.02313 \\
\hline 
\end{tabular} 
\label{tab:lions_dd_1}
\end{table}

\section{A posteriori error analysis for Cumulative Distribution Function (CDF) computations} \label{sec_posteriori}

In this section, we construct a posteriori error estimates of the error in the computed distribution of
a given QoI. The estimate takes into account stochastic sources of error arising from finite sampling and deterministic sources arising from discretization of the differential equation.

\subsection{Approximating the CDF}

\tpurp{The solution $u$ depends implicitly on a random vector $\bs \theta$ and hence the QoI, $Q(u)$, is a random variable.}
We approximate the CDF,
\begin{equation*}
P(t) = P\big(\{\bs \theta: Q(u^n(\bs \theta)) \leq t\}\big) = P\big(Q \leq t\big),
\end{equation*}
using a finite number of approximate sample values $\big\{Q(U^n)\}_{i=1}^\mathcal{N}$:
\begin{equation*}
 P_N(t) = \frac{1}{\mathcal{N}} \sum_{n=1}^\mathcal{N} I\big( Q(U^n)\leq t \big),
\end{equation*}
where $I$ is the indicator function. The formal Monte Carlo solution algorithm is given in Algorithm~\ref{alg:mc}.

\begin{minipage}{0.7\textwidth}
\begin{algorithm}[H]
	Draw samples $\{\bs \theta^n\}_{n=1}^\mathcal{N}$ from the distribution of $\bs \theta$ \\
	\For{$n = 1, \cdots, \mathcal{N}$ (number of samples)}{
		Compute solutions $\{U^n\}$ to produce samples $\{Q(U^n)\}$\\
	}
	Approximate the output distribution using a standard nonparametric technique, e.g. via binning
\caption{Formal Monte Carlo algorithm}
\label{alg:mc}
\end{algorithm}
\end{minipage}

\subsection{Motivating examples}
\label{sec:motivating_example}
In general, there is a balance between the error arising from finite
sampling and discretization that should be struck for efficiency. The
results are often surprising in the sense that despite the slow
convergence of the Monte Carlo method, numerical discretization error
is often the most significant source of error. We refer to \cite{estep2008nonparametric,estep2009nonparametric}
for further discussion. \tpurp{For the case of uniform refinement, we define the normalized mesh parameter, $\tilde{h}$ as the ratio: ($h$ of given mesh / $h$ of coarsest mesh). In the numerical examples $\tilde{h}=1.0$ corresponds to a mesh with 249 vertices while $\tilde{h}=0.5$ corresponds to a mesh of 945 vertices.}

\subsubsection{Poisson Equation}
\label{sec:motivating_poisson}
We illustrate with an the Poisson equation from \S \ref{sec:num_ex_trans_dom_effect}  with $a=1$ and $f(x,y) = 200x(1-x)+200y(1-y)$. For the QoI, we choose $\psi = \chi 10 x y$ in \eqref{eq:qoi_orig_def} where $\chi$ is the characteristic function of $[0.50, 0.75] \times [0.50,0.75]$. The nominal reference domain is a unit square $[0,1]\times[0,1]$. The boundary points are perturbed according to a uniform distribution $[-0.08 , 0.08] \times [-0.08 , 0.08]$ centered around each boundary point.  The random perturbations are graphically illustrated in Figure~\ref{exp_1_fig_0_a}, where each boundary point is sampled uniformly in the square around each boundary point. Examples of two perturbed domains are shown in Figures~\ref{exp_1_fig_0_b} and \ref{exp_1_fig_0_c}.
\begin{figure}[htbp]
	\begin{center}
	\subfigure[]{\includegraphics[width=0.3\textwidth]{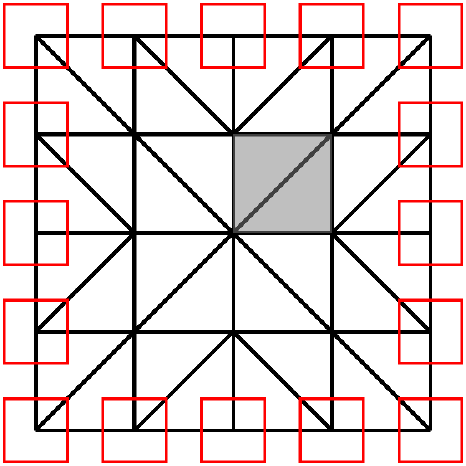}\label{exp_1_fig_0_a}}
	\subfigure[]{\includegraphics[width=0.3\textwidth]{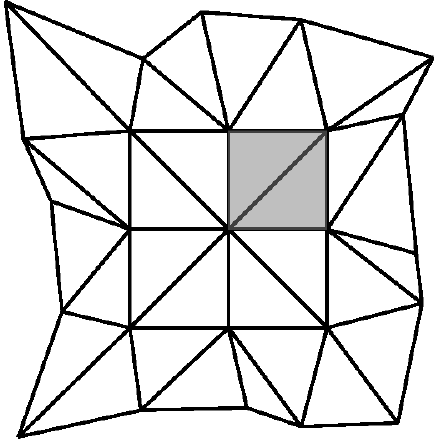}\label{exp_1_fig_0_b}}
	\subfigure[]{\includegraphics[width=0.3\textwidth]{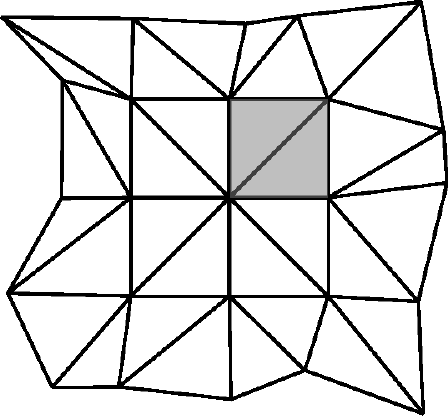}\label{exp_1_fig_0_c}}
	\caption{(a) Random domains are obtained by perturbing the boundary points uniformly $0.08$ in each direction as indicated by the squares around each point. (b) and (c) Two randomly perturbed domains.} 
	\label{exp_1_fig_0}
	\end{center}
\end{figure}

Two mesh configurations are used, a coarse mesh with normalized mesh parameter $\tilde{h} = 1.0$ and a finer mesh with $\tilde{h}=0.5$. \tpurp{The ranges of the 2-norm condition numbers of the linear systems arising from the finite element discretizations  are $[70.65, \  275.97]$ and $[318.378, \  1466.17]$ corresponding to $\tilde{h} = 1.0$ and $\tilde{h} = 0.5$ respectively.}
To approximate the true error, we compute a \emph{reference} distribution using a fine discretization with $\tilde{h} =.25$ and 10,000 samples. We plot the distributions obtained using fewer samples, $N = 100$ and $N = 1000$, and coarser meshes, $\tilde{h} = 1$ and $\tilde{h} = 0.5$ in Fig.~\ref{exp_1_fig_1}.
\begin{figure}[htbp]
	\begin{center}
		 \subfigure[]{\includegraphics[width=0.4\textwidth]{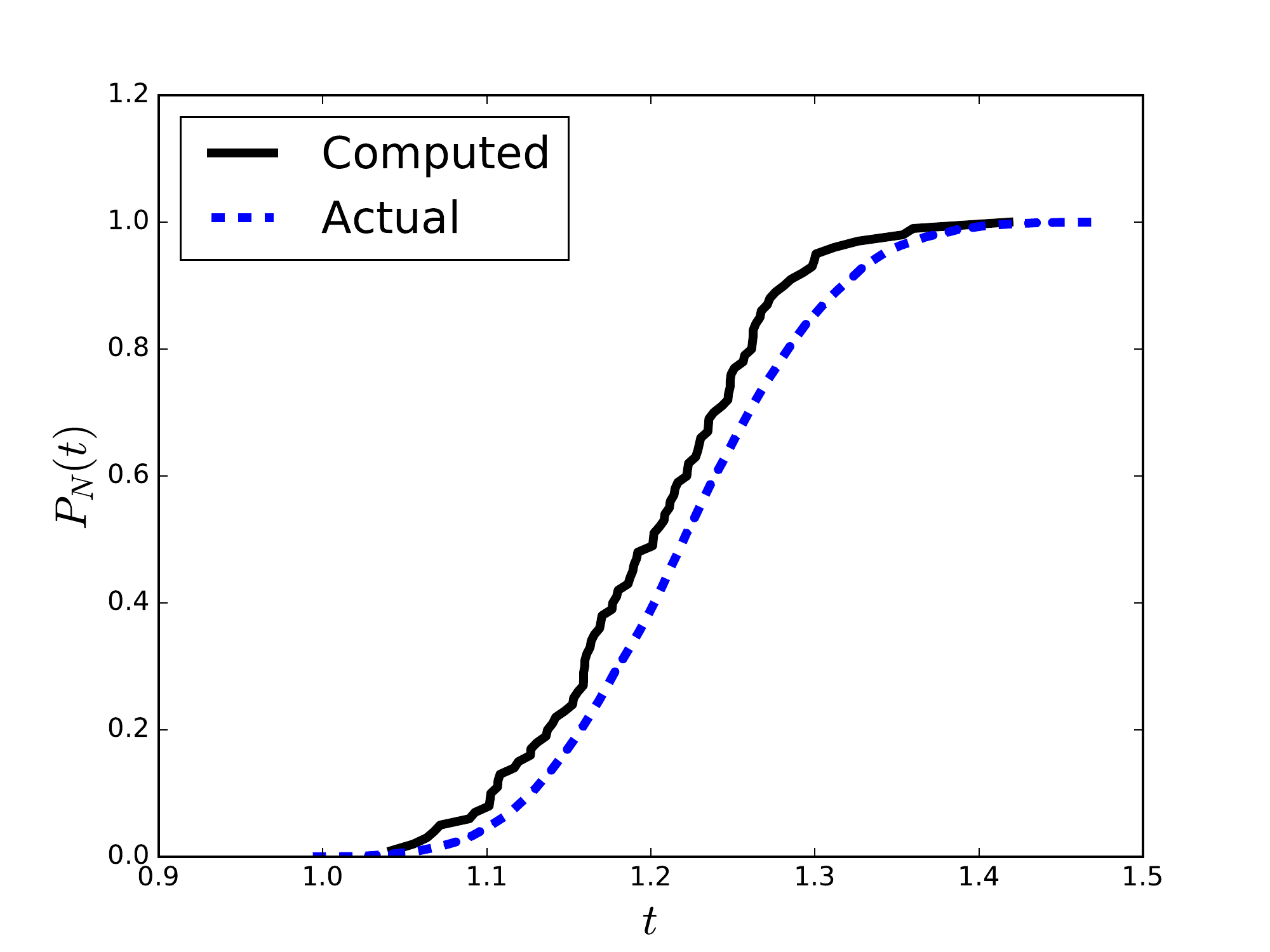}} 
		 \subfigure[]{\includegraphics[width=0.4\textwidth]{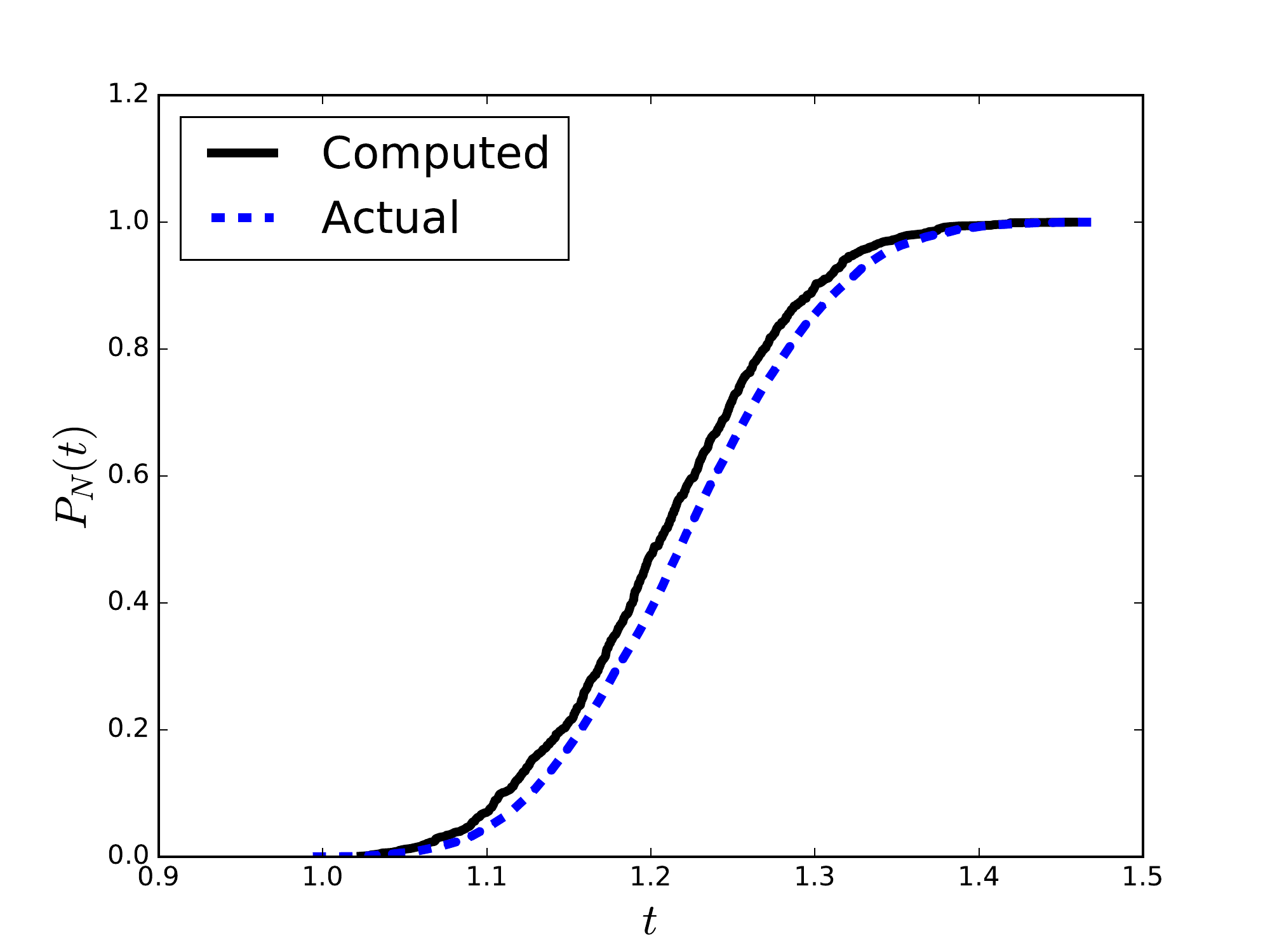}}\\
		 \subfigure[]{\includegraphics[width=0.4\textwidth]{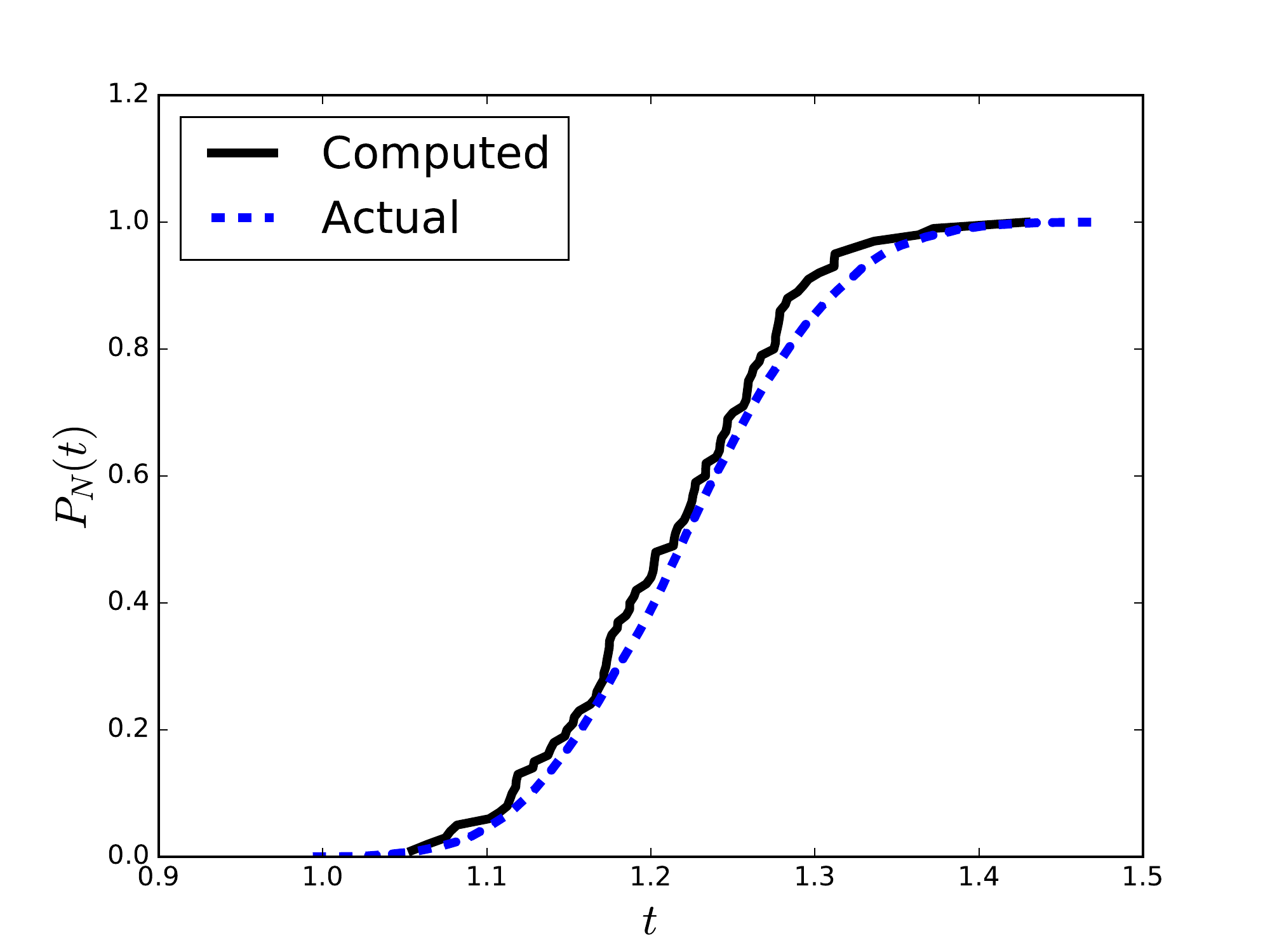}} 
		 \subfigure[]{\includegraphics[width=0.4\textwidth]{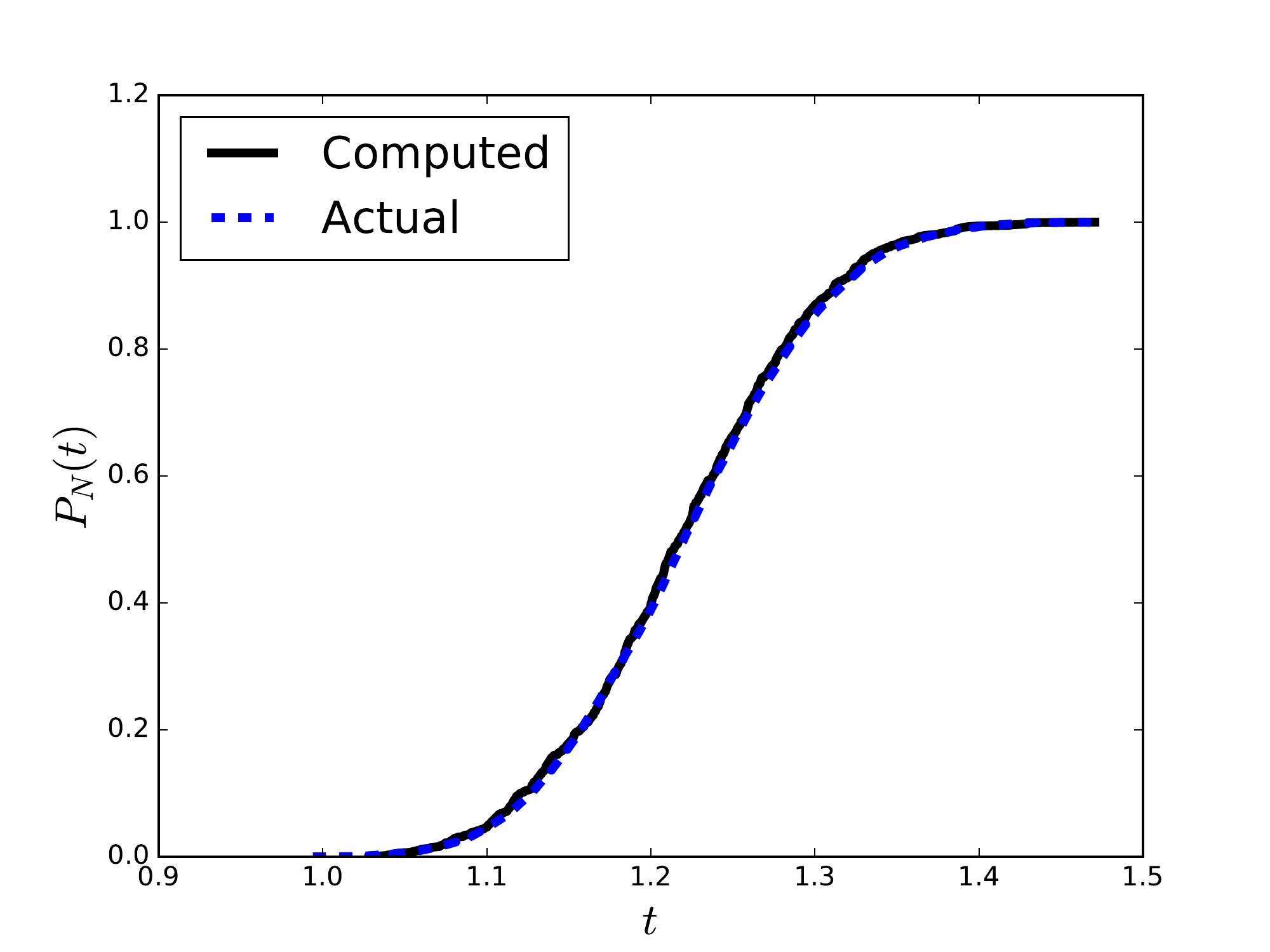}}
		\caption{
		The reference distribution (dashed line) plotted with several approximate empirical distributions. The normalized mesh parameters are $\tilde{h} = 1$ (left) and $\tilde{h} = 0.5$ (right).  The number of samples are $N = 100$ (top) and $N = 1000$ (bottom).} \label{exp_1_fig_1}
	\end{center}
\end{figure}

These results demonstrate that the discretization and finite sampling can result in significant bias and variance in the estimated distributions.  Variance can be reduced effectively by increasing $N$, while reducing bias involves refining the mesh for the differential equation solve.  The a posteriori error estimate provides a way to quantity the contributions of finite sampling and discretization to the error in the computed distribution.

\subsubsection{Convection-Diffusion Equation}
\label{sec:motivating_cd}

We consider the equation,
\begin{equation}
\left\{
	\begin{aligned}
   -\nabla \cdot \left( a(x) \nabla w(x;\bs\theta) \right)  + \mathbf{b}(x) \cdot \nabla w &= f(x), \quad && x \in \Omega(\bs \theta),
\\ \
   w(x;\bs\theta) &= 0, \quad && x \in \partial\Omega(\bs \theta).
   \end{aligned} \right.
   \label{eq:cd_random_domain}
\end{equation}
Here $\mathbf{b}$ is a convective vector field while rest of the variables have the same meaning as in \eqref{model_random_domain}. The transformed problem reads,
\begin{equation}
\footnotesize
\left\{
\begin{aligned}
&-\nabla \cdot \pars{ \mathbf{A}^n(y;\bs\theta^n) \nabla u^n(y;\bs\theta^n) } + \hat{\mathbf{b}}^n(y;\bs\theta^n) \cdot  \nabla u^n(y;\bs\theta^n) = F^n(y;\bs\theta^n), &y \in \Omega,\\
&u^n(y;\bs\theta^n) = 0,  \qquad &y \in \partial\Omega.
\end{aligned} \right.
\label{eq:cd_trans_domain}
\end{equation}
Here $\hat{\mathbf{b}}^n(y;\bs\theta^n) = \hat{\mathbf{b}}^n_d(y;\bs\theta^n)$ for $y \in \Omega_d$ defined  by,
\[\hat{\mathbf{b}}^n_d(y;\bs\theta^n) = \lvert \mathrm{det} \mathbf{J}_d^n \lvert ^{-1} \; \mathbf{J}_d^n \; \mathbf{b}(\varphi_d^n)^{-1}(y)), \qquad  y \in \Omega_d
\] 
while rest of the variables have similar definitions as in \S~\ref{sec:transformed_problem_poisson}.

We choose $\mathbf{b} = [-80, \, 0]^\top$ is the convection vector field. Other parameters are set as $a=1$, $f(x,y) = 200\sin(2\pi x)\sin(2\pi y)$. For the QoI, we choose $\psi = \chi 10 x y$ in \eqref{eq:qoi_orig_def} where $\chi$ is the characteristic function of $[0.50, 0.75] \times [0.50,0.75]$. We plot the distributions obtained using  $N = 100$ and $N = 1000$ samples, and $\tilde{h} = 1$ and $\tilde{h} = 0.5$ in Fig.~\ref{exp_cd_fig_1}. 
\tpurp{The ranges of the 2-norm condition numbers of the linear systems arising from the finite element discretizations  are $[12.8628, \  46.1225]$ and $[51.654, \  224.635]$ corresponding to $\tilde{h} = 1.0$ and $\tilde{h} = 0.5$ respectively.}
The reference distribution is computed using $4000$ samples and $\tilde{h}=0.125$.
\begin{figure}[htbp]
	\begin{center}
		 \subfigure[]{\includegraphics[width=0.4\textwidth]{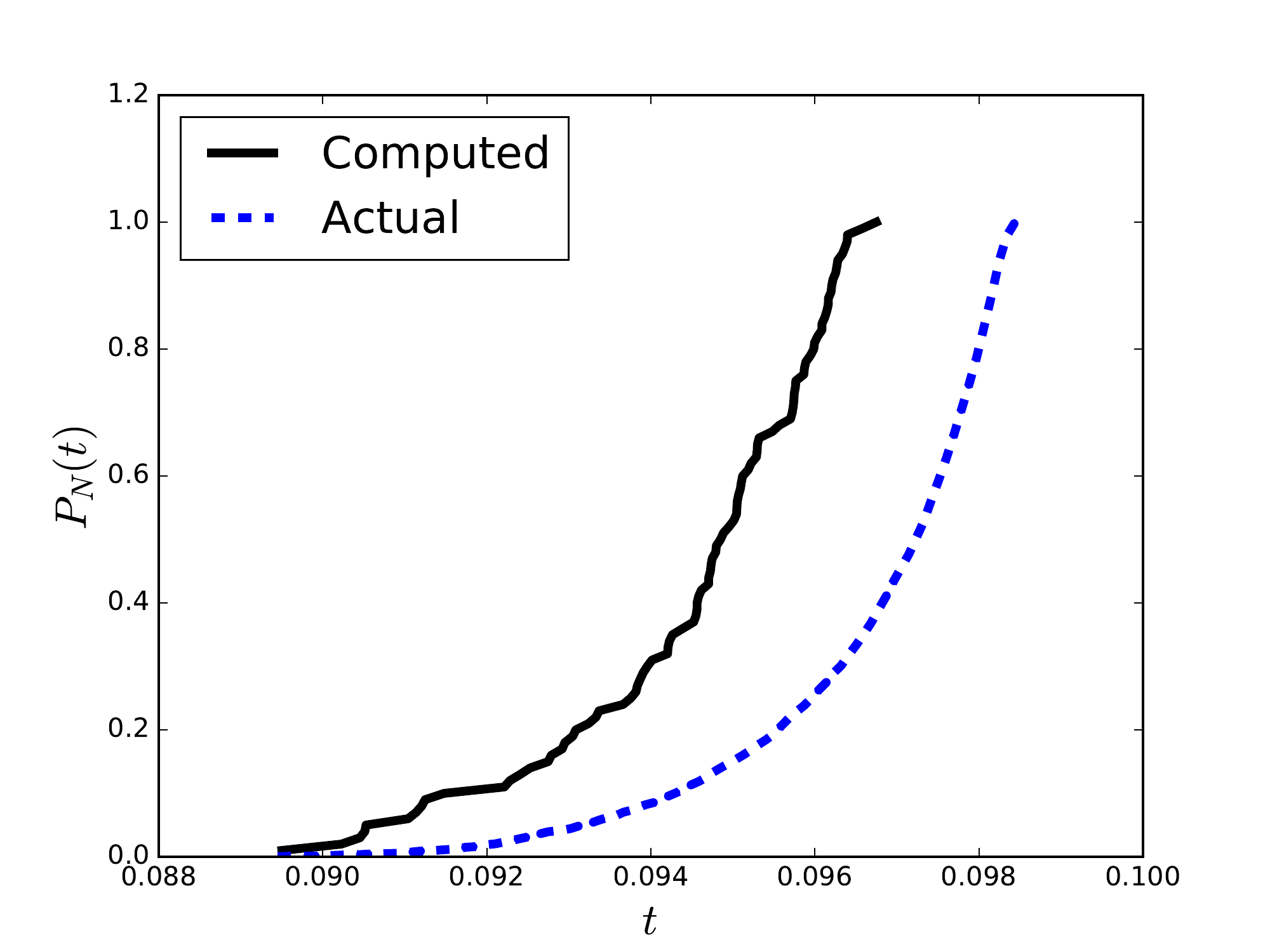}} 
		 \subfigure[]{\includegraphics[width=0.4\textwidth]{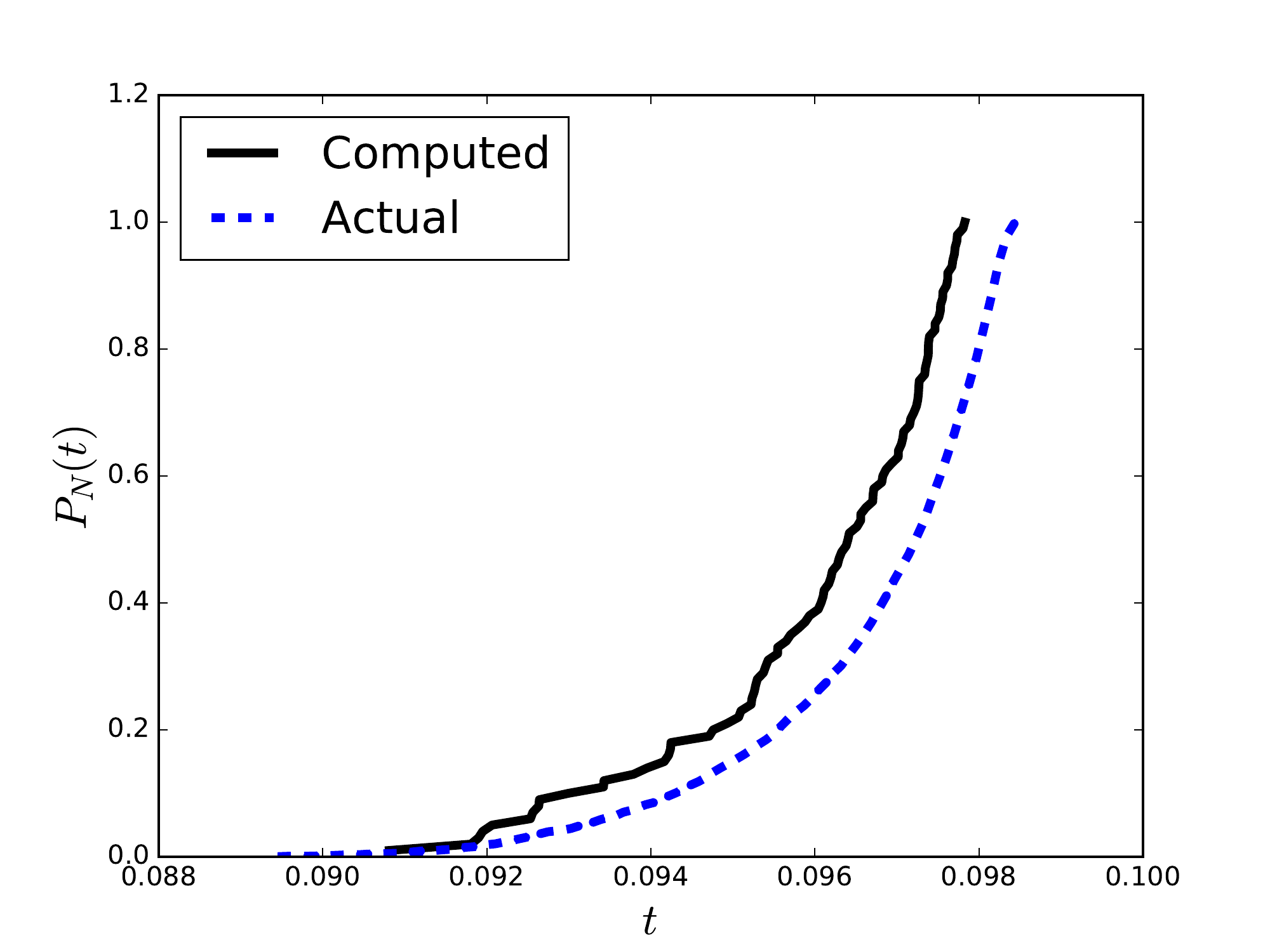}}\\
		 \subfigure[]{\includegraphics[width=0.4\textwidth]{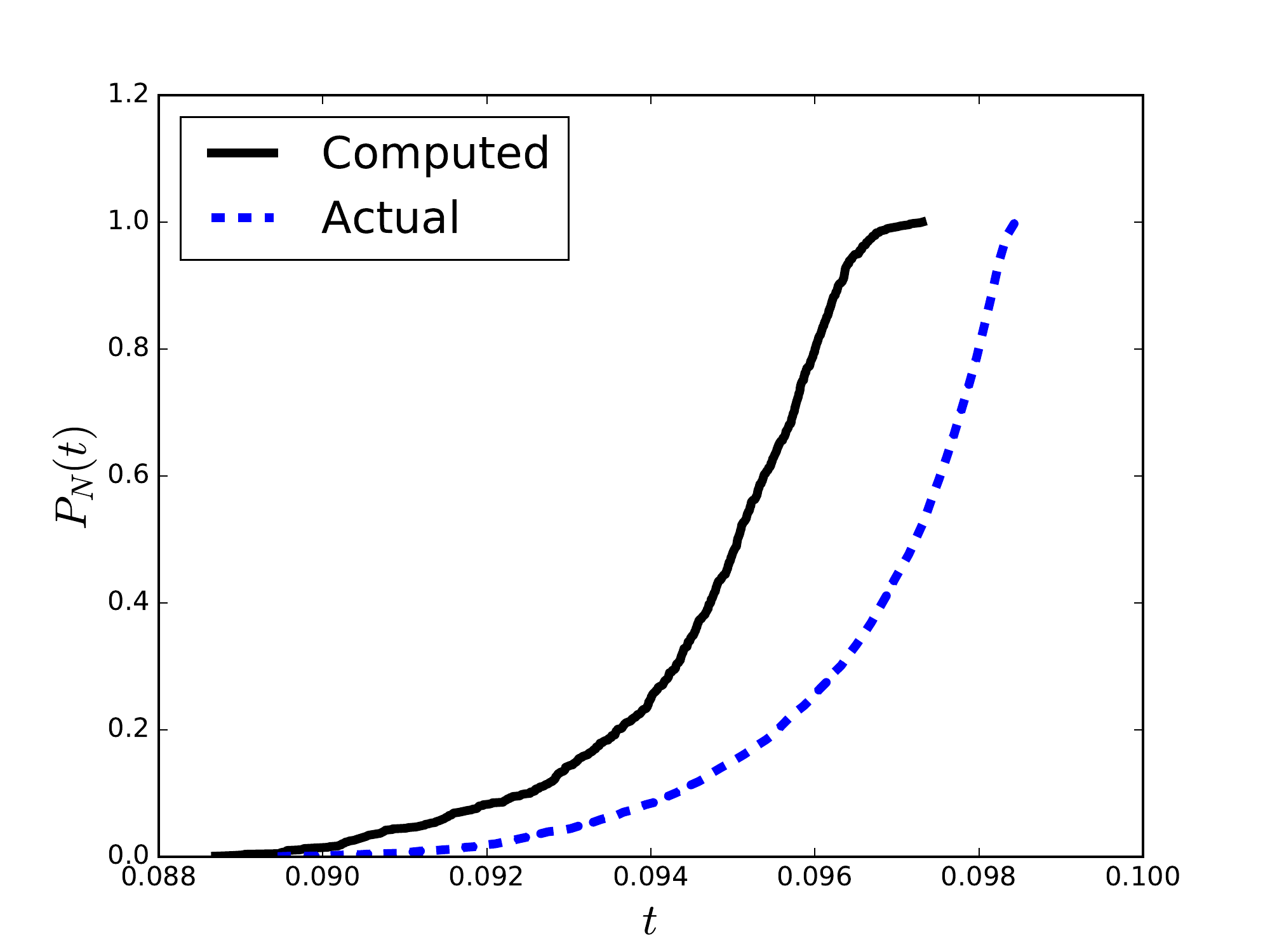}} 
		 \subfigure[]{\includegraphics[width=0.4\textwidth]{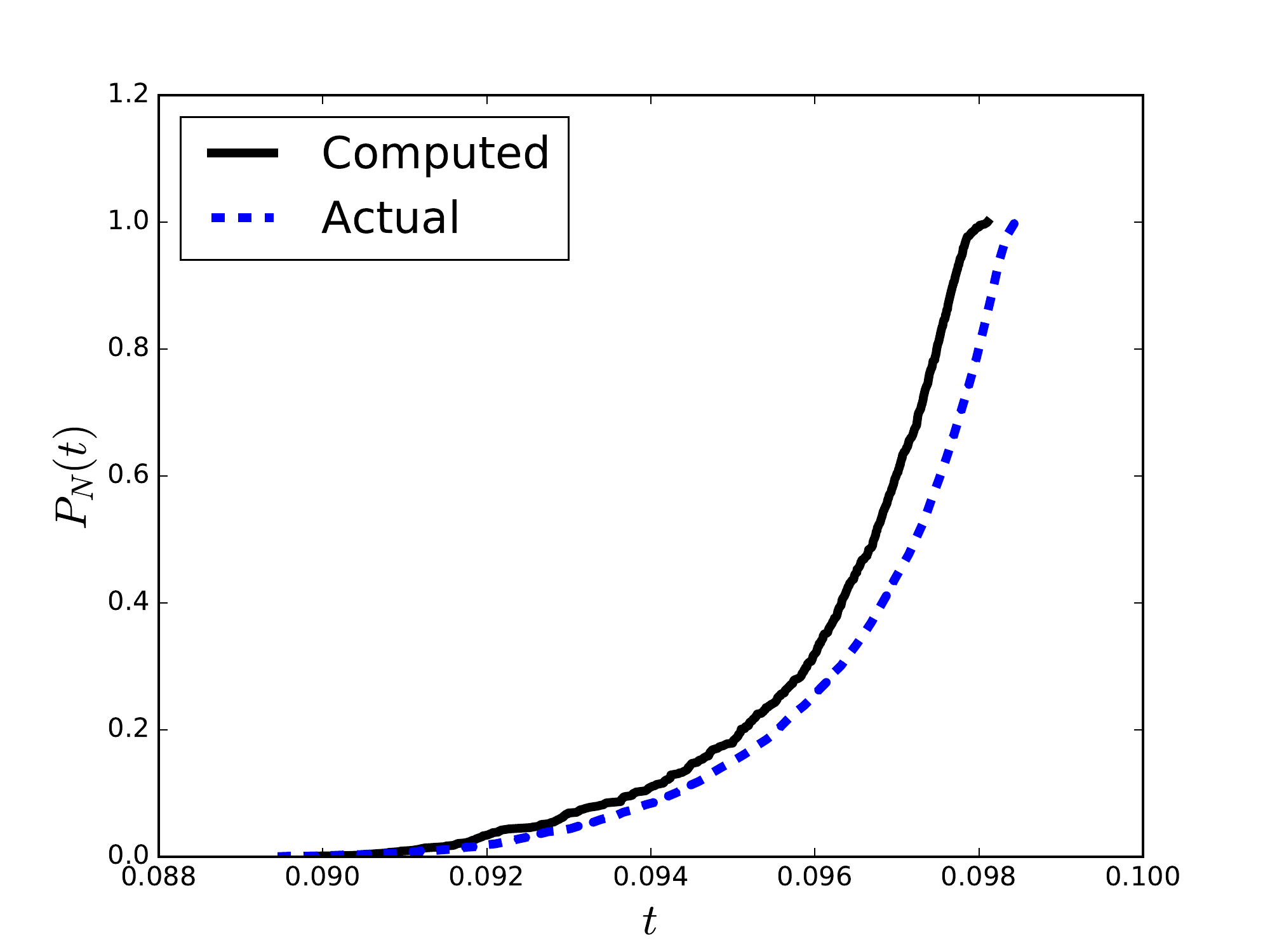}}
		\caption{
		The reference distribution (dashed line) plotted with several approximate empirical distributions. From left to right, distributions are computed using $\tilde{h} = 1$, $\tilde{h} = 0.5$,  The number of samples are $N = 100$ (top) and $N = 1000$ (bottom).}
		 \label{exp_cd_fig_1}
	\end{center}
\end{figure}

These results numerically demonstrate the claim that numerical error in CDF computations is dominated by discretization errors. Refining the spatial mesh once has a huge impact on reducing the error in the computed CDF, whereas increasing the samples has a relatively minor effect in computing a more accurate CDF.
\subsection{A posteriori error analysis for a computed distribution function}

We approximate $P(t)$ by the approximate sample distribution function,
\begin{align}
	\wh{P}_N(t) = \frac{1}{N} \sum_{n=1}^N \mathbf{1}_{(-\infty,t]}(Q(\wt{U}^{n};\bs\theta^n)), \notag
\end{align}
where $\{\wt{U}^n\}$ are the solutions corresponding to samples $\{ \bs\theta^n \}_{n=1}^N$ and $\mathbf{1}_S$ denotes the indicator function for the set $S$.  For the purpose of error analysis, we introduce the ``nominal'' sample distribution using exact model solves by,
\begin{align}
	P_N(t) = \frac{1}{N} \sum_{n=1}^N \mathbf{1}_{(-\infty,t]}(Q(u;\bs\theta^n)), \notag
\end{align}
and decompose the error into contributions arising from finite sampling and discretization of the differential equation,
\begin{align}
	P(t) - \wh{P}_N(t)
		  = (P(t) - P_N(t)) + (P_N(t) - \wh{P}_N(t)). \notag
\end{align}
This decomposition is used to derive the following error bound
\begin{theorem}[\cite{estep2008nonparametric,estep2009nonparametric}]{}  For $0 < \varepsilon < 1$,
\begin{align}
	|P(t) - \wh{P}_N(t)|
		&\leq \pars{ \frac{\wh{P}_N(t)(1-\wh{P}_N(t))}{N \varepsilon}}^{1/2}
			+ \frac{2}{N} \left| \sum_{n=1}^N \mathbf{1}_{[-|\mathcal{E}^n|,|\mathcal{E}^n|]}(t-Q(\wt{U}^{n};\bs\theta^n)) \right|
				+ \frac{1}{2N\varepsilon}, \label{error_in_dist}
\end{align}
with probability greater than or equal to $1-\varepsilon$, where $\mathcal{E}^n $ is an error estimate for the QoI computed from the sample numerical solution, i.e.,
\begin{align}
	\mathcal{E}^n \approx Q(u^{n};\bs\theta^n) - Q(\wt{U}^{n};\bs\theta^n) . \label{error_estimate_assumed}
\end{align}
\end{theorem}
The first term on the right of \eqref{error_in_dist} is a standard bound on error arising from finite sampling. There is a discussion of such bounds in \cite{estep2008nonparametric,estep2009nonparametric} as well as computations illustrating its accuracy.  The second term quantifies bias in the computed distribution arising from numerical error as ``shifts'' in the distribution function. Evaluating the second term requires a computational error estimate for each sample value.

\subsection{Numerical experiments}

\subsubsection{Poisson Equation}
\label{sec:num_exp_cdf_poisson}
We present a numerical experiment showing the behavior of the complete error estimate as the number of samples is increased and the finite element mesh is refined. Figure \ref{exp_2_fig_1_a} shows the error bound and the actual error in the CDF for the example in \S \ref{sec:motivating_poisson} with $\epsilon=0.05$. The error bound is around six times larger than the actual error. Figure \ref{exp_2_fig_1_b}  shows the stochastic and the discretization contributions. Increasing the number of samples or refining the mesh both decrease the error in the CDF as observed in Figure \ref{exp_2_fig_2}. The figure also indicates that adding more samples decreases the variance, making the CDF smoother, while refining the mesh primarily targets the discretization contribution. 

\begin{figure}[htbp]
	\begin{center}
		 \subfigure[]{\includegraphics[width=0.42\textwidth]{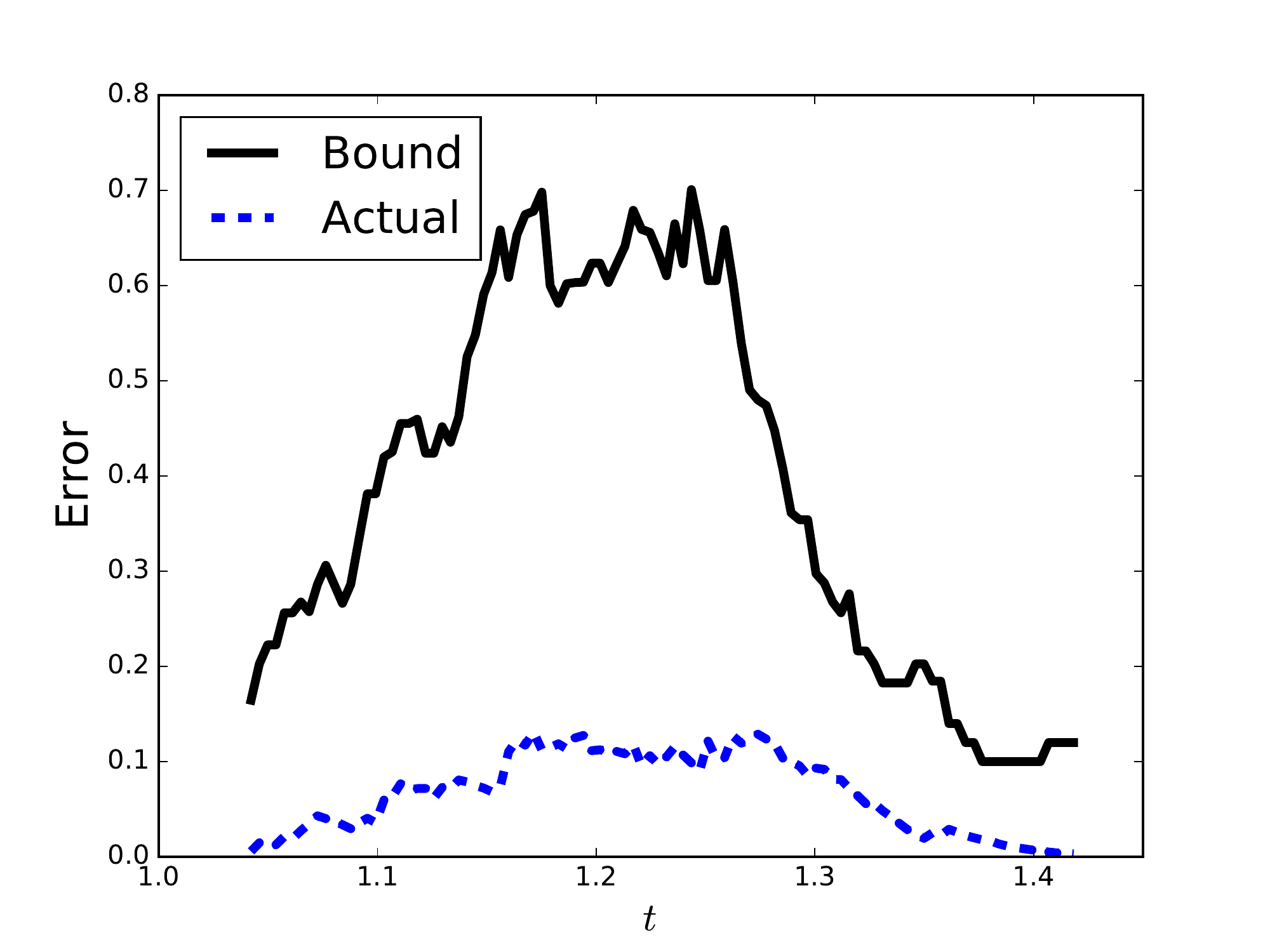} \label{exp_2_fig_1_a}} 
		 \subfigure[]{\includegraphics[width=0.42\textwidth]{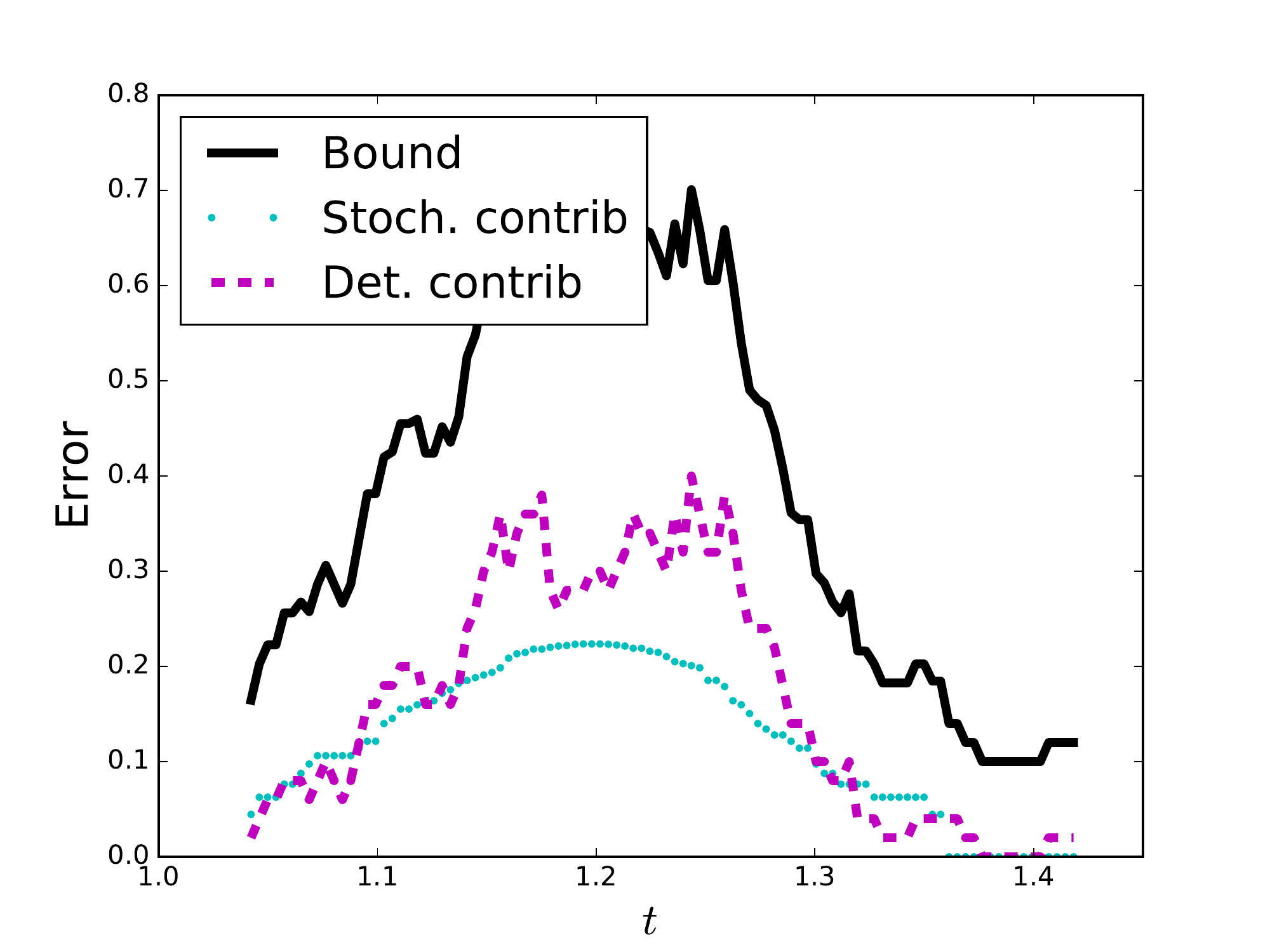} \label{exp_2_fig_1_b}} 
		\caption{
		(a) Actual error and bound for distribution  computed using $\tilde{h} = 1, N = 100$. (b) The solid, dashed and dotted lines indicate the total error bound, the discretization contribution and the stochastic contributions.}
		\label{exp_2_fig_1}
	\end{center}
\end{figure}

\begin{figure}[htbp]
	\begin{center}
 \subfigure[]{\includegraphics[width=0.42\textwidth]{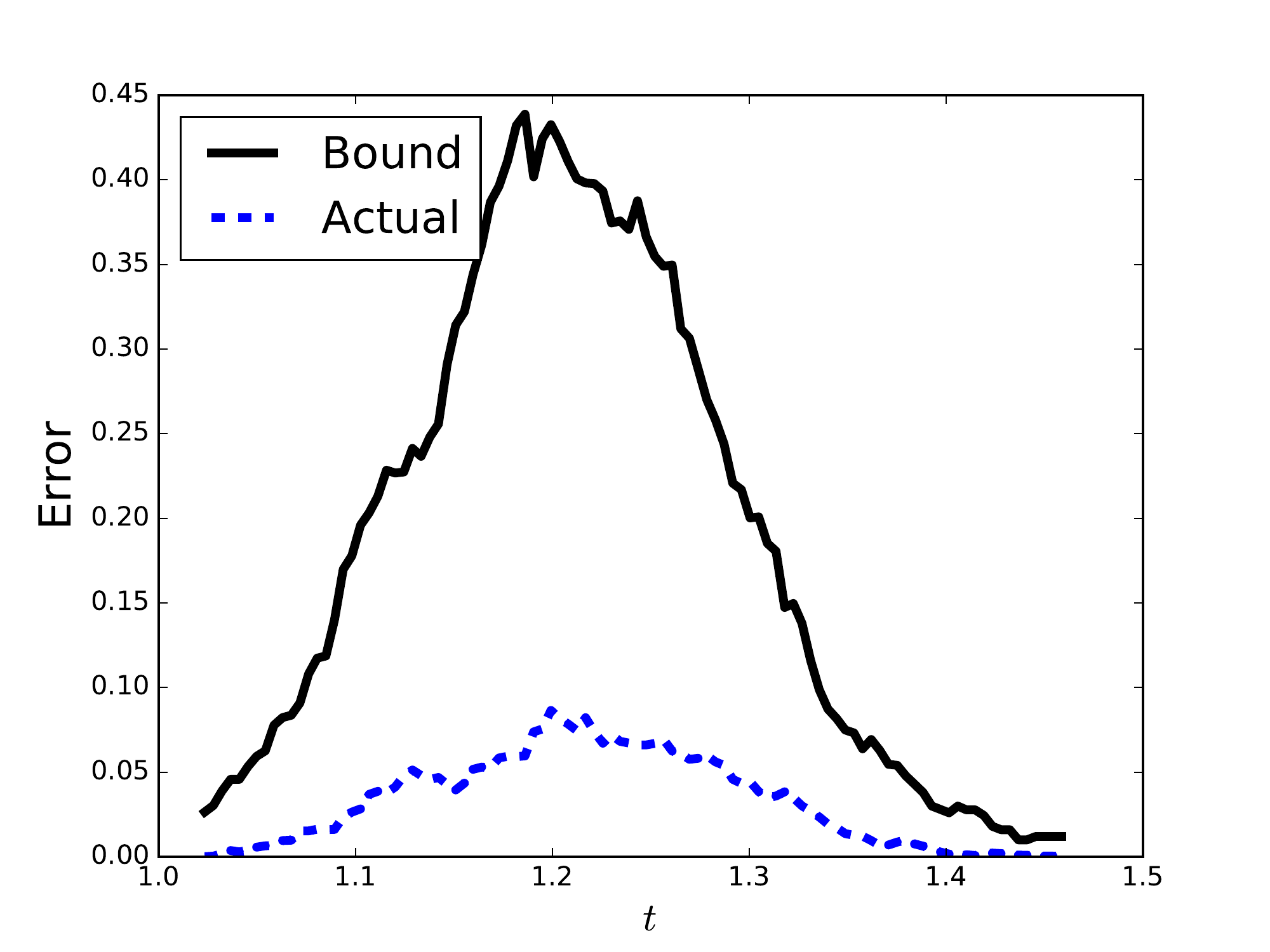} \label{exp_2_fig_2_a}} 
 \subfigure[]{\includegraphics[width=0.42\textwidth]{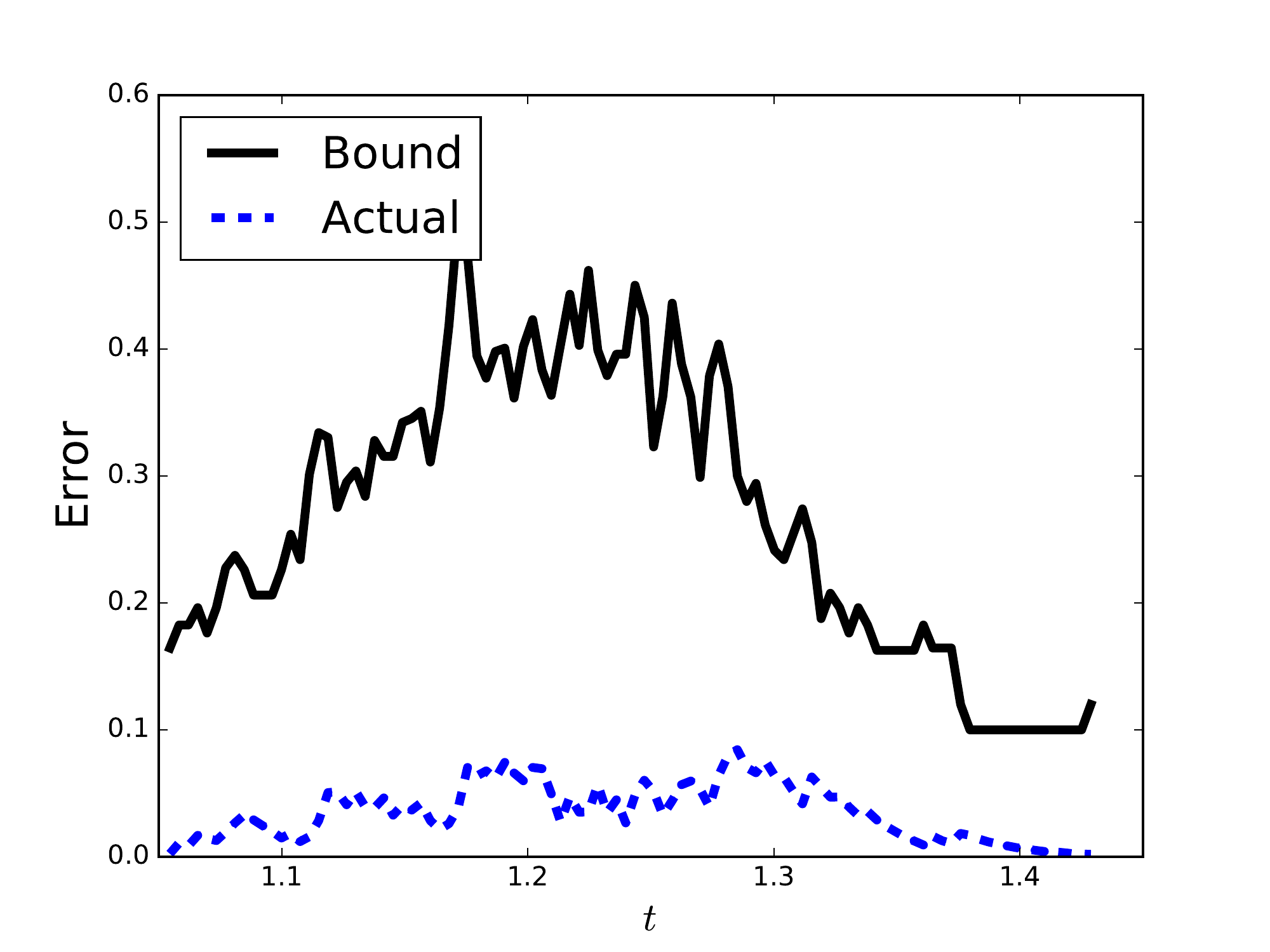} \label{exp_2_fig_2_b}}
\caption{
		 Actual error and bound. From left to right, distributions are computed using  $\tilde{h} = 1, N = 1000$ and $\tilde{h} = 0.5, N = 100$. } 
		\label{exp_2_fig_2}
	\end{center}
\end{figure}	

\subsubsection{Convection-Diffusion Equation}
\label{sec:num_exp_cdf_cd}

A similar experiment is performed for the convection-diffusion problem in \S~\ref{sec:motivating_cd}. The weak form of the  adjoint equation to \eqref{eq:cd_trans_domain} is
\begin{equation*}
   \int_{\Omega} \mathbf{A}^n\nabla \eta^n \cdot \nabla v  + \hat{\mathbf{b}}^n\cdot \nabla v \, \eta^n \, \mathrm{d}y = \int_{\Omega} \psi , v \, \mathrm{d}y, \quad \forall v \in H_0^1(\Omega)
\end{equation*}
The corresponding error representation  (compare to \eqref{eq:err_est_trans}) is
\begin{equation}
\label{eq:err_est_trans_cd}
Q(e^n; \bs \theta) = \int_\Omega F^n \eta^n \, \mathrm{d}y - \int_\Omega \mathbf{A}^n \nabla U^n \cdot \nabla \eta^n \,\mathrm{d}y - \hat{\mathbf{b}}^n\cdot \nabla U^n \, \eta^n.
\end{equation}

Figure \ref{exp_cd_fig_2_a} shows the error bound and the actual error in the CDF. The error bound is around six times larger than the actual error. Figure \ref{exp_cd_fig_2_b}  shows the stochastic and the discretization contributions.

\begin{figure}[htbp]
	\begin{center}
		 \subfigure[]{\includegraphics[width=0.42\textwidth]{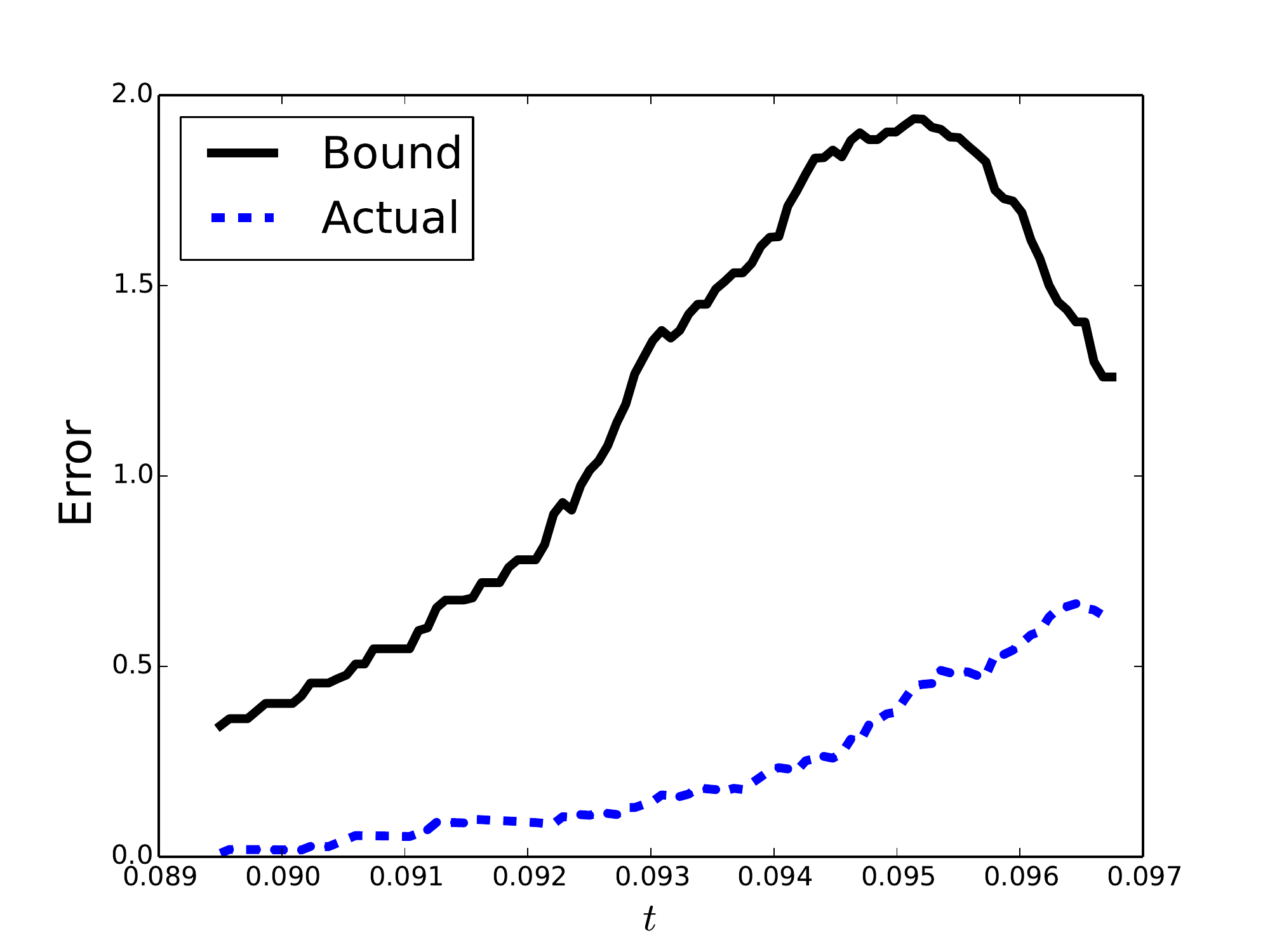} \label{exp_cd_fig_2_a}} 
		 \subfigure[]{\includegraphics[width=0.42\textwidth]{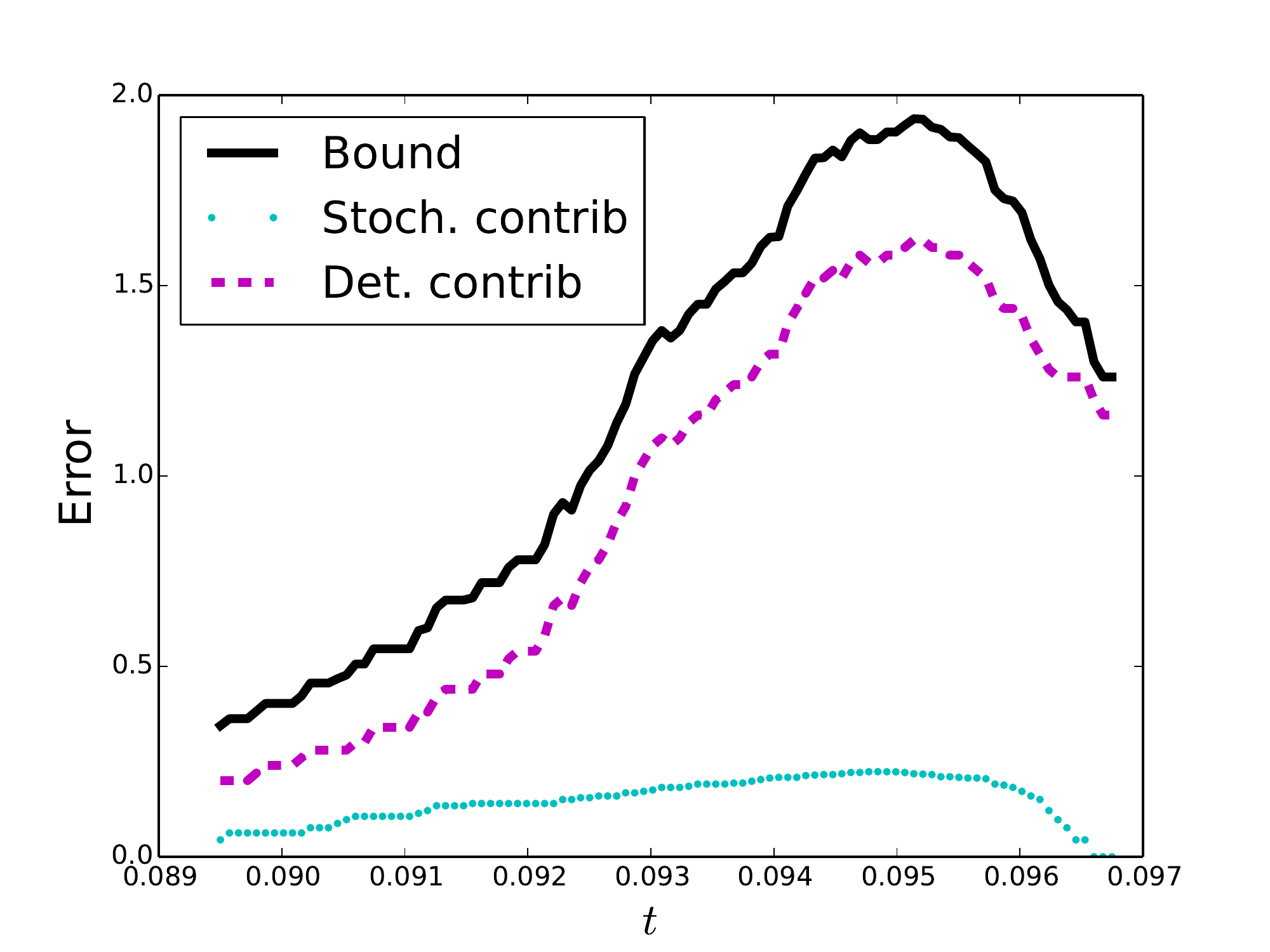} \label{exp_cd_fig_2_b}} 
		\caption{
		(a) Actual error and bound for distribution  computed using $\tilde{h} = 1, N = 100$. (b) The solid, dashed and dotted lines indicate the total error bound, the discretization contribution and the stochastic contributions.}
		\label{exp_cd_fig_2}
	\end{center}
\end{figure}

\section{Adaptive Mesh Refinement}
\label{sec:AMR}
We note that the use of a different triangulation $\mathcal{T}_h^n$ for each sample $\bs\theta^n$ is indicated in situations in which the solution behavior - and therefore numerical accuracy - depends on the shape of the domain $\Omega(\bs\theta^n)$. On the other hand, meshing is often a computationally demanding aspect of discretization, so there is incentive to employ one mesh for all samples. As an alternative to the use of a computationally inefficient heavily refined uniform mesh, we now describe an adaptive algorithm  that attempts to optimize across samples.

A posteriori error estimates are well-suited to guide adaptive mesh refinement algorithms \cite{estep2000estimating}. The general approach is iterative: Start with a coarse mesh, compute the solution, evaluate the estimate, use the estimate to select elements for refinement, refine the mesh, and then iterate.

Nominally, a straightforward adaptive algorithm based on the a posteriori
error estimate would  require computing a sequence of adaptive meshes for each sample computation. This would appear to be inefficient however.  We describe a ``greedy'' adaptive
algorithm that seeks to find an adaptive mesh that works for all
realizations at the cost of some loss of efficiency. The idea is to
construct an adapted mesh for the first realization. We use the
adapted mesh as the initial mesh for the computation of the second realization. The initial mesh is refined (but not coarsened) as necessary to control the error in the second sample.  This repeats, where at each step the
mesh is further refined and is described in Algorithm~\ref{find_mesh_algorithm}. In practice, we find that within 10-50 iterations, the algorithm constructs a mesh that works for all subsequent realizations.

\begin{minipage}{0.7\textwidth}
\begin{algorithm}[H]
	Start with an initial (uniform) mesh $\mathcal{T}_h^0$ \\
	\For{$n = 1,\hdots,M$ (loop through the set of $M$ test problems)}{
		$\mathcal{T}_h^n =\mathcal{T}_h^{n-1}$ \\
		Compute solution $U^n$ and error estimate $\mathcal{E}^n$\\
		\While{$|\mathcal{E}^n| > \mathrm{TOL}$}{	  
				
		  	Refine the mesh, i.e., $\mathcal{T}_h^n = \text{refine}(\mathcal{T}_h^{n})$ \\
		  	Compute solution $U^n$ and error estimate $\mathcal{E}^n$} 

	}
\caption{Constructing a  universal finite element mesh}
\label{find_mesh_algorithm}
\end{algorithm}
\end{minipage}

Algorithm~\ref{find_mesh_algorithm} is demonstrated for the problem in \S \ref{sec:num_exp_cdf_cd}. The per element error indicator $E_K$ measures the error on element $K$ and is defined as
\begin{equation}
\label{eq:err_ind}
E_K = \left| \int_K F^n \eta^n  -  \mathbf{A}^n \nabla U^n \cdot \nabla \eta^n \,\mathrm{d}y - \hat{\mathbf{b}}^n\cdot \nabla U^n \, \eta^n \right|.
\end{equation}
The D\"{o}rfler strategy is used for marking elements with large elemental error indicators~\cite{Do96}. If the QoI error at a parameter value is above a specified tolerance, which is taken to be 0.0004 for this example, then the mesh is adapted for that parameter till the error falls below the tolerance, $\mathrm{TOL}$.
The mesh is only refined for four parameter values, and after iteration 28, there is no further refinement for the remaining 972 iterations as shown in Figure \ref{exp_adaptive_fig_1}. \tpurp{The refinement pattern indicates that the mesh is refined around the region $[0.50, 0.75] \times [0.50,0.75]$ where the support of the QoI function $\psi$ lies and to the right of this region. The mesh is not refined around all the boundary nodes even though they are all perturbed. This is because  the error indicator \eqref{eq:err_ind} describes the sensitivity of the QoI to different parts of the domain and for this example the refined region contributes the most to the error. The results may be explained heuristically by noting that the direction of the vector field $\mathbf{b}$ is from right to left, and hence the QoI is affected the most from the region to its right.}

\begin{figure}[htbp]
	\begin{center}
		 \subfigure[]{\includegraphics[width=0.31\textwidth]{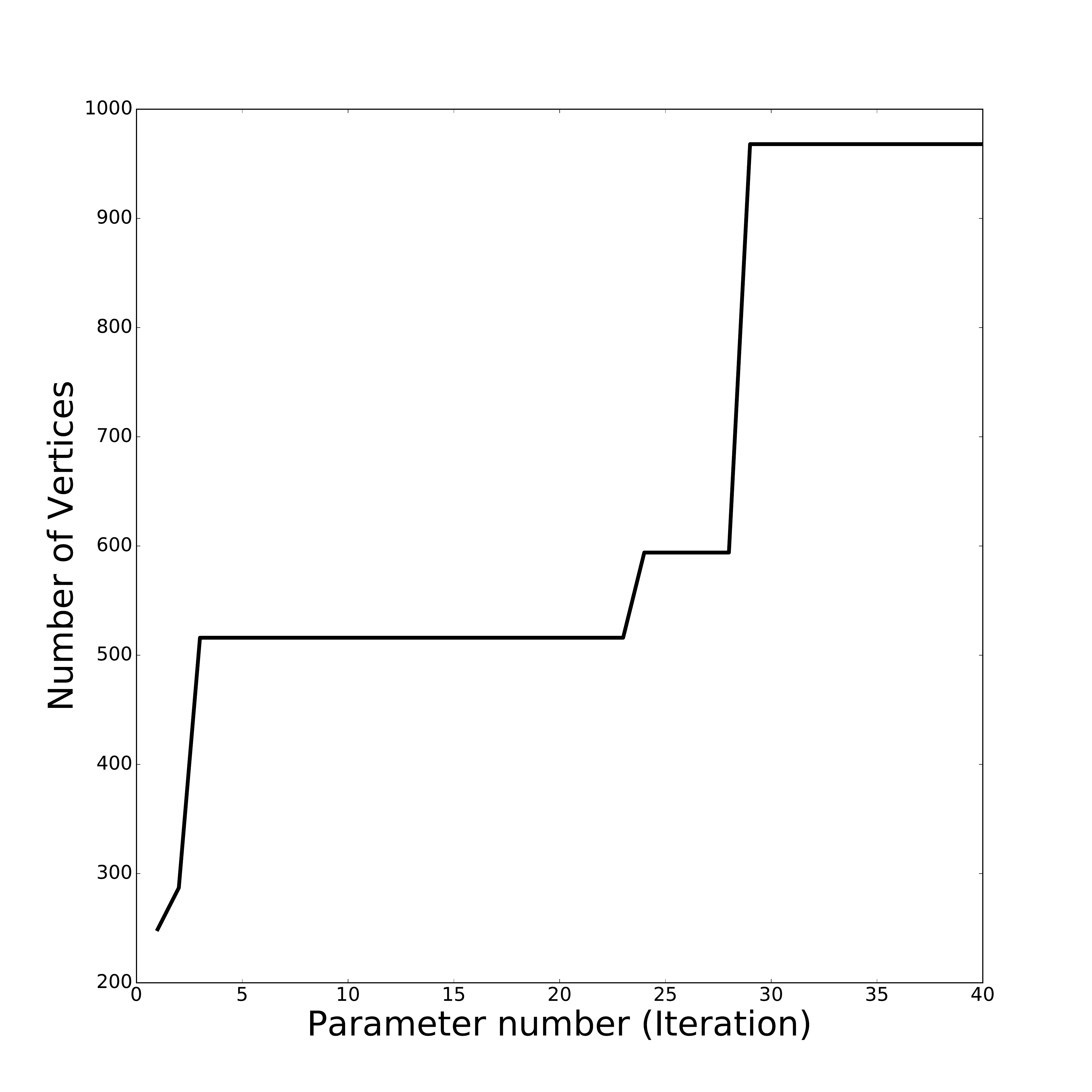}}
		\subfigure[iteration 0]{\includegraphics[width=0.31\textwidth]{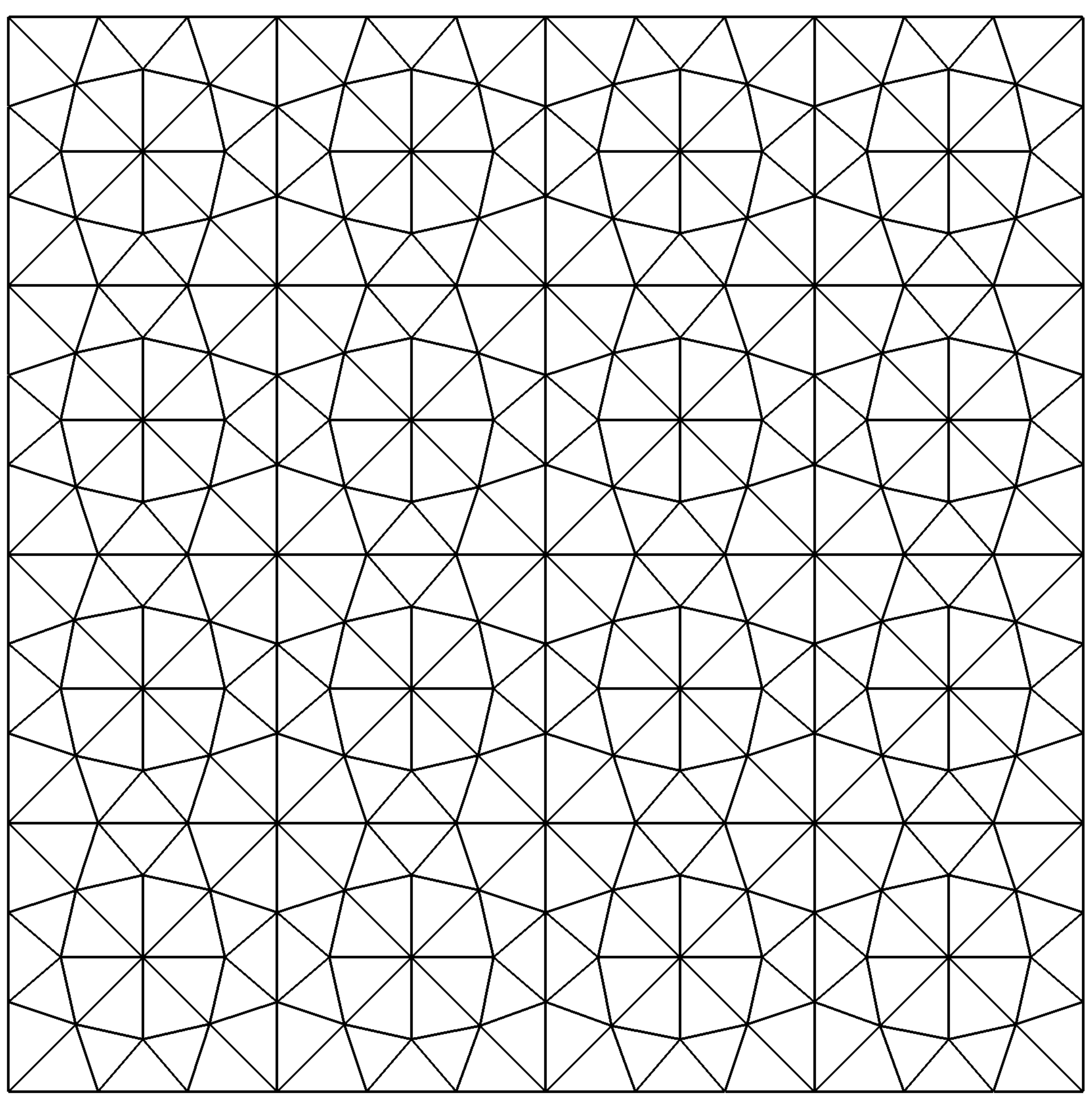}}
		\subfigure[iteration 41]{\includegraphics[width=0.31\textwidth]{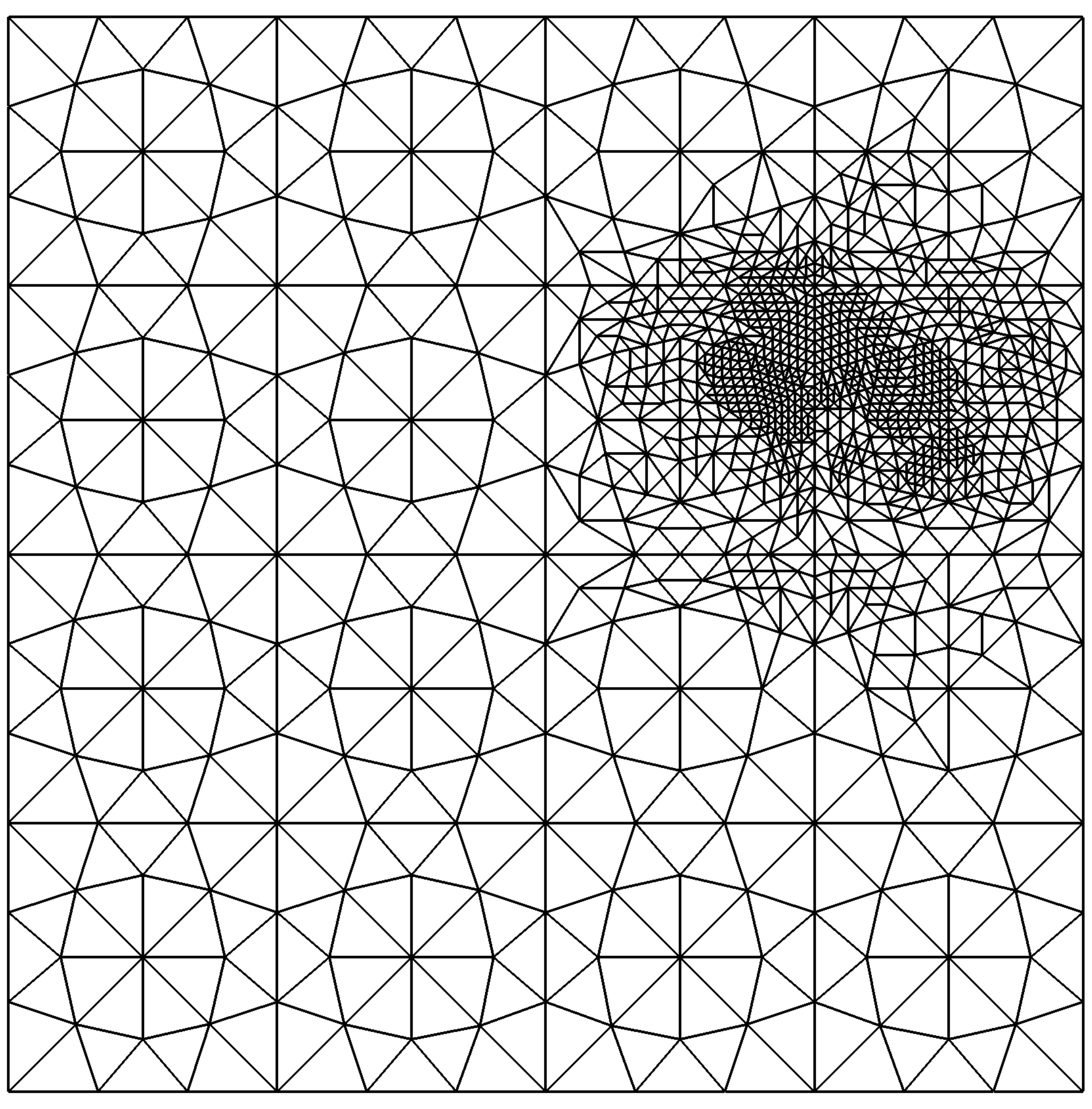}}
		\caption{
		In panel (a) The number of vertices in the finite element mesh is plotted against parameter number.
		Panels (b) and (c) show the initial mesh and the final refined mesh respectively.} \label{exp_adaptive_fig_1}
	\end{center}
\end{figure}

The CDF computed using the adaptive algorithm with 1000 samples is plotted in Figure~\ref{exp_adaptive_fig_2_a}. This CDF compares well with the CDF computed a twice uniformly refined mesh corresponding to $\tilde{h}=0.25$ with 4000 samples (labeled ``Actual'' in the plot). The adaptive algorithm is carried out on a mesh of 992 vertices (after 41 iterations) compared to 945 vertices for the mesh corresponding to $\tilde{h}=0.5$. Although the number of vertices in the meshes are similar, and hence the computational effort is similar, the CDF computed using the adaptive method is significantly more accurate. The dramatic increase in the accuracy of the CDF is explained by observing the error contributions in Figure~\ref{exp_adaptive_fig_2}. Now the discretization contribution is significantly than the CDF computed with no refinement, e.g. see Figure~\ref{exp_cd_fig_2_b}.
\begin{figure}[htbp]
	\begin{center}
		\subfigure[]{\includegraphics[width=0.42\textwidth]{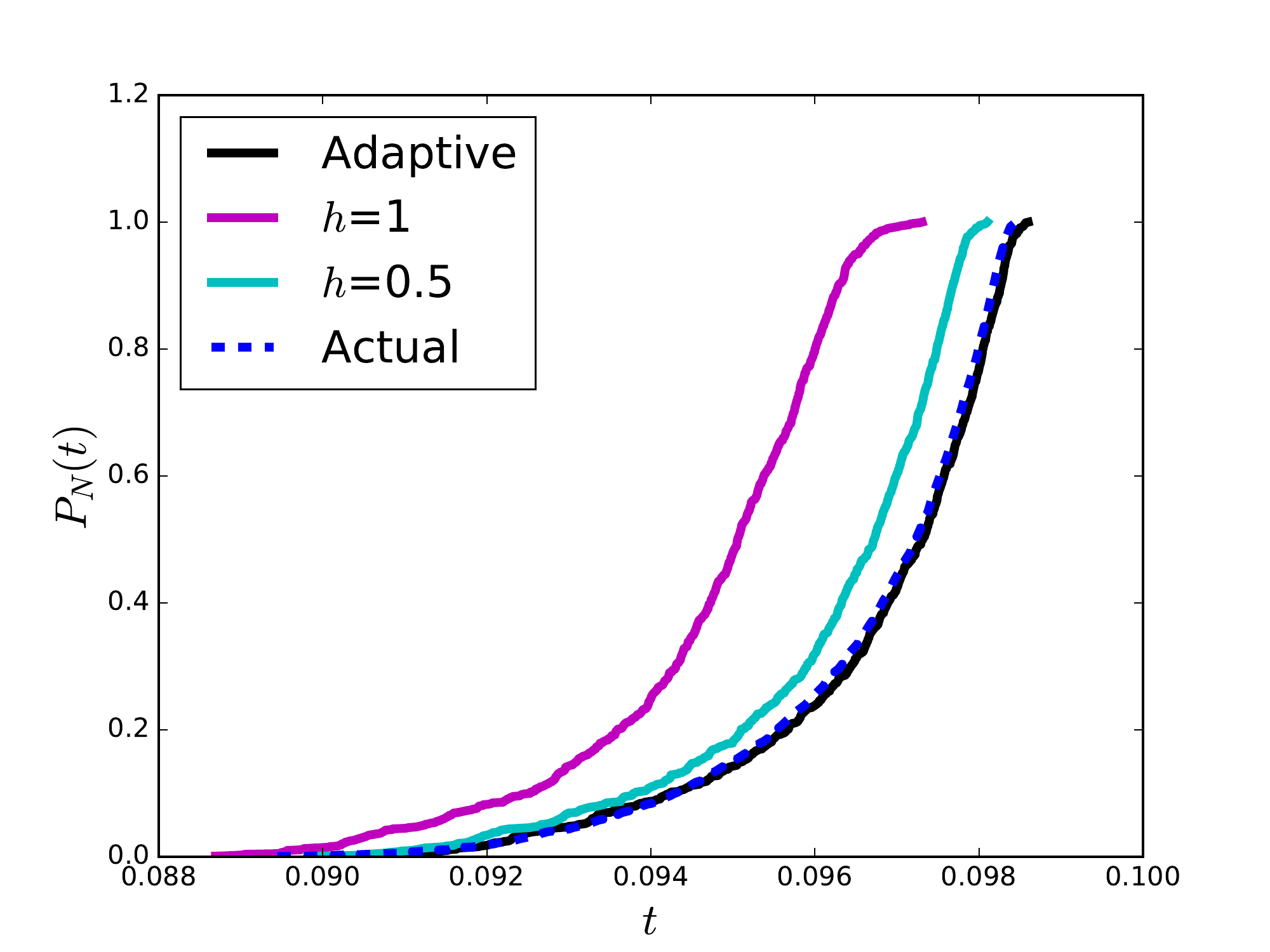}\label{exp_adaptive_fig_2_a}}	
		\subfigure[]{\includegraphics[width=0.42\textwidth]{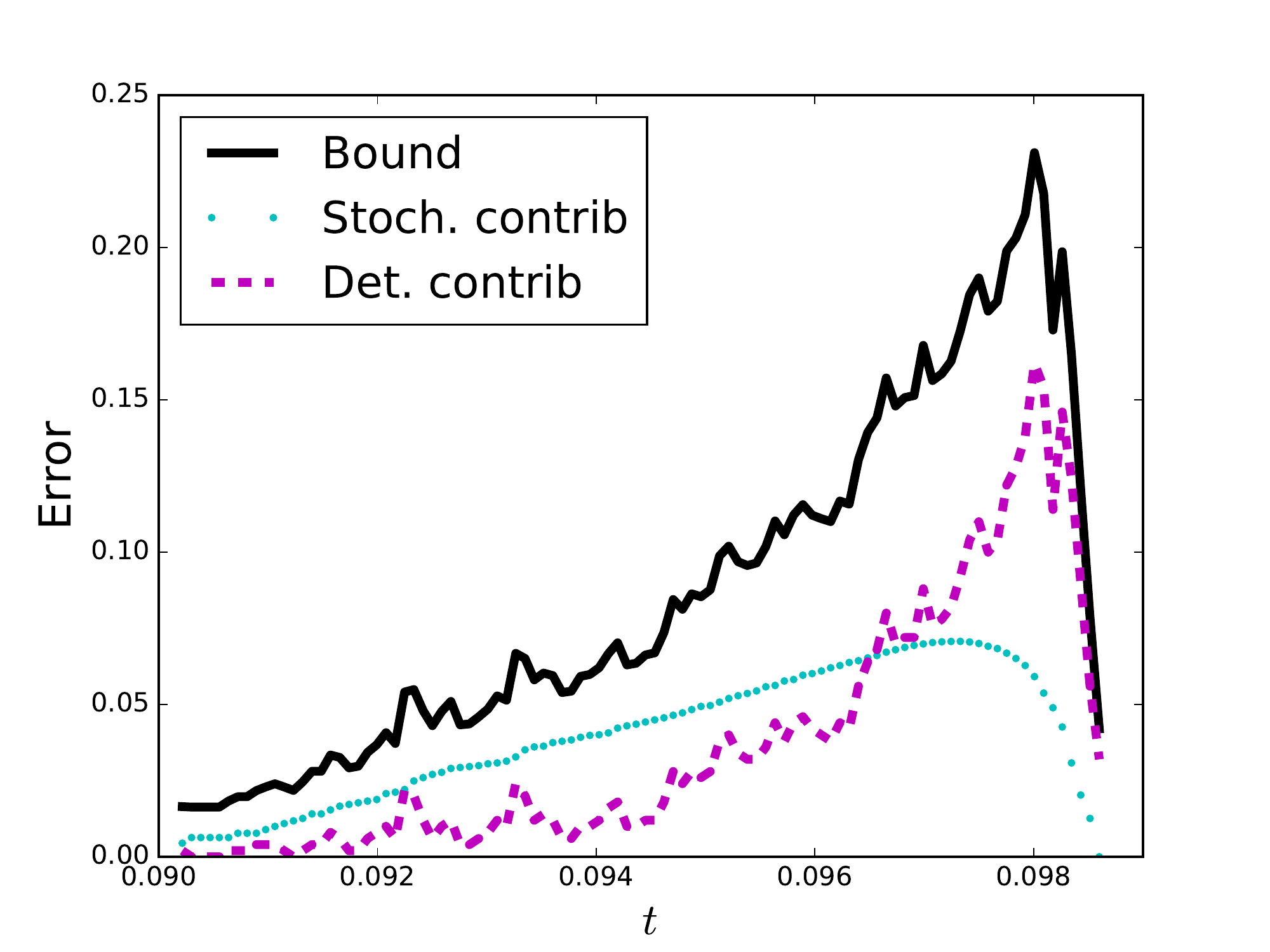}\label{exp_adaptive_fig_2_b}}
		\caption{
		 (a) CDF computed using adaptive refinement with 1000 samples (solid line) against the CDF computed using a $\tilde{h} = 1, N = 1000$, $\tilde{h}=0.5,N=1000$ and ``Actual'' CDF computed thrice uniformly refined mesh ($\tilde{h} = 0.125$) and 4000 samples.
		(b) Error bound and contributions.} \label{exp_adaptive_fig_2}
	\end{center}
\end{figure}

\section{Summary}

We study elliptic partial differential equation boundary value problems for which the domain where the equations
are posed is uncertain. In order to model this uncertainty, we formulate a class of problems that are posed on stochastic domains.
Of particular interest is the nonparametric density estimation problem for a quantity of interest.
We introduce a piecewise transformation of the domain to a deterministic reference domain
and exploit for efficiency the transformation in a Monte Carlo sampling procedure so that many samples
can be obtained to approximate the distribution at a reasonable cost.
We present an a posteriori error analysis for each sample and for the empirical distribution function obtained from the samples,
which reflect the various deterministic and statistical sources of error.
The estimate is sufficiently detailed that we can efficiently balance computational
work, e.g. mesh resolution versus sample numbers, to achieve a desired
accuracy. An interesting issues that arises in the context of numerical solution
is that each realization of a domain nominally requires construction of a new
discretization mesh, at a significant computational cost.
We address computational efficiency by describing an adaptive
strategy that leads to a mesh that produces acceptable accuracy for
all realizations of the problem and describing an iterative
solution algorithm in which the number of matrix inversions is
independent of the number of samples.

\bibliographystyle{plain}
\bibliography{random_domain_bib_2}

\begin{thebibliography}{10}

\bibitem{arnst2009probabilistic}
M.~Arnst and R.~Ghanem.
\newblock {Probabilistic Electromechanical Modeling of Nanostructures with
  Random Geometry}.
\newblock {\em Journal of Computational and Theoretical Nanoscience},
  6(10):2256--2272, 2009.

\bibitem{babuska1970finite}
I.~Babu{\v{s}}ka.
\newblock The finite element method for elliptic equations with discontinuous
  coefficients.
\newblock {\em Computing}, 5(3):207--213, 1970.

\bibitem{babuska2005worst}
I.~Babu{\v{s}}ka, F.~Nobile, and R.~Tempone.
\newblock {Worst case scenario analysis for elliptic problems with
  uncertainty}.
\newblock {\em Numerische Mathematik}, 101(2):185--219, 2005.

\bibitem{babuska2004galerkin}
I.~Babu{\v{s}}ka, R.~Tempone, and G.~E. Zouraris.
\newblock {Galerkin finite element approximations of stochastic elliptic
  partial differential equations}.
\newblock {\em Siam J. Numer. Anal}, 42(2):800--825, 2004.

\bibitem{bjorstad1986iterative}
P.~E. Bjorstad and O.~B. Widlund.
\newblock Iterative methods for the solution of elliptic problems on regions
  partitioned into substructures.
\newblock {\em SIAM Journal on Numerical Analysis}, pages 1097--1120, 1986.

\bibitem{blyth2003heat}
M.~G. Blyth and C.~Pozrikidis.
\newblock {Heat conduction across irregular and fractal-like surfaces}.
\newblock {\em International Journal of Heat and Mass Transfer},
  46(8):1329--1339, 2003.

\bibitem{brady1993diffusive}
M.~Brady and C.~Pozrikidis.
\newblock {Diffusive transport across irregular and fractal walls}.
\newblock {\em Proceedings: Mathematical and Physical Sciences},
  442(1916):571--583, 1993.

\bibitem{bramble1986iterative}
J.~H. Bramble, J.~E. Pasciak, and A.~H. Schatz.
\newblock An iterative method for elliptic problems on regions partitioned into
  substructures.
\newblock {\em Math. Comput.}, 46(174):361--370, 1986.

\bibitem{broyda2010probability}
S.~Broyda, M.~Dentz, and D.~M. Tartakovsky.
\newblock {Probability density functions for advective--reactive transport in
  radial flow}.
\newblock {\em Stochastic Environmental Research and Risk Assessment}, pages
  1--8, 2010.

\bibitem{canuto2009numerical}
C.~Canuto and D.~Fransos.
\newblock {Numerical solution of partial differential equations in random
  domains: an application to Wind Engineering}.
\newblock {\em Citeseer}, 2009.

\bibitem{CNT16}
Julio~E. Castrillón-Candás, Fabio Nobile, and Raúl~F. Tempone.
\newblock Analytic regularity and collocation approximation for elliptic pdes
  with random domain deformations.
\newblock {\em Computers \& Mathematics with Applications}, 71(6):1173 -- 1197,
  2016.

\bibitem{ciarlet1978finite}
P.G. Ciarlet.
\newblock {\em The finite element method for elliptic problems}, volume~4.
\newblock North-Holland, 1978.

\bibitem{Do96}
W.~D\"{o}rfler.
\newblock A convergent adaptive algorithm for {P}oisson's equation.
\newblock {\em SIAM Journal on Numerical Analysis}, 33(3):1106--1124, 1996.

\bibitem{driscoll2002schwarz}
T.~A. Driscoll and L.~N. Trefethen.
\newblock {\em {Schwarz-christoffel mapping}}.
\newblock Cambridge University Press Cambridge, 2002.

\bibitem{eehj_actanum_95}
K.~Eriksson, D.~Estep, P.~Hansbo, and C.~Johnson.
\newblock Introduction to adaptive methods for differential equations.
\newblock In {\em Acta Numerica, 1995}, Acta Numerica, pages 105--158.
  Cambridge Univ. Press, Cambridge, 1995.

\bibitem{estep2008nonparametric}
D.~Estep, A.~Malqvist, and S.~Tavener.
\newblock {Nonparametric density estimation for randomly perturbed elliptic
  problems I: computational methods, a posteriori analysis, and adaptive error
  control}.
\newblock {\em SIAM Journal on Scientific Computing}, 31:2935--2959, 2009.

\bibitem{estep2009nonparametric}
D.~Estep, A.~Malqvist, and S.~Tavener.
\newblock {Nonparametric density estimation for randomly perturbed elliptic
  problems II: Applications and adaptive modeling}.
\newblock {\em International Journal for Numerical Methods in Engineering},
  80:846–--867, 2009.

\bibitem{estep2000estimating}
D.~J. Estep, M.~G. Larson, R.~D. Williams, and American~Mathematical Society.
\newblock {\em Estimating the error of numerical solutions of systems of
  reaction-diffusion equations}.
\newblock American Mathematical Society, 2000.

\bibitem{GR09}
C.~Geuzaine and J.-F. Remacle.
\newblock Gmsh: A 3-d finite element mesh generator with built-in pre- and
  post-processing facilities.
\newblock {\em International Journal for Numerical Methods in Engineering},
  79(11):1309--1331, 2009.

\bibitem{harbrecht2010output}
H.~Harbrecht.
\newblock {On output functionals of boundary value problems on stochastic
  domains}.
\newblock {\em Mathematical Methods in the Applied Sciences}, 33(1):91--102,
  2010.

\bibitem{HPS16}
H.~Harbrecht, M.~Peters, and M.~Siebenmorgen.
\newblock Analysis of the domain mapping method for elliptic diffusion problems
  on random domains.
\newblock {\em Numerische Mathematik}, 134(4):823--856, Dec 2016.

\bibitem{HSSCS15}
R.~Hiptmair, L.~Scarabosio, C.~Schillings, and C.~Schwab.
\newblock Large deformation shape uncertainty quantification in acoustic
  scattering.
\newblock {\em Research Report, ETH Zurich Switzerland}, 2015-31, 2015.

\bibitem{larsson2008partial}
S.~Larsson and V.~Thom{\'e}e.
\newblock {\em {Partial differential equations with numerical methods}}.
\newblock Springer Verlag, 2008.

\bibitem{lions1988schwarz}
P.~L. Lions.
\newblock On the schwarz alternating method. i.
\newblock In {\em First international symposium on domain decomposition methods
  for partial differential equations}, pages 1--42, 1988.

\bibitem{lions1990schwarz}
P.~L. Lions.
\newblock On the schwarz alternating method iii: a variant for nonoverlapping
  subdomains.
\newblock In {\em Third international Symposium on domain decomposition methods
  for partial differential equations}, volume~6, pages 202--223. SIAM:
  Philadelphia, PA, 1990.

\bibitem{nouy2008extended}
A.~Nouy, A.~Clement, F.~Schoefs, and N.~Mo{\.e}s.
\newblock {An extended stochastic finite element method for solving stochastic
  partial differential equations on random domains}.
\newblock {\em Computer Methods in Applied Mechanics and Engineering},
  197(51-52):4663--4682, 2008.

\bibitem{oba2010global}
R.~M. Oba.
\newblock {Global boundary flattening transforms for acoustic propagation under
  rough sea surfaces}.
\newblock {\em The Journal of the Acoustical Society of America}, 128:39, 2010.

\bibitem{tartakovsky2006stochastic}
D.~M. Tartakovsky and D.~Xiu.
\newblock {Stochastic analysis of transport in tubes with rough walls}.
\newblock {\em Journal of Computational Physics}, 217(1):248--259, 2006.

\bibitem{trefethen1979numerical}
L.~N. Trefethen.
\newblock {Numerical Computation of the Schwarz-Christoffel Transformation.},
  1979.

\bibitem{xiu2007efficient}
D.~Xiu and J.~Shen.
\newblock {An efficient spectral method for acoustic scattering from rough
  surfaces}.
\newblock {\em Communications in computational physics}, 2:54--72, 2007.

\bibitem{xiu2007numerical}
D.~Xiu and D.~M. Tartakovsky.
\newblock {Numerical methods for differential equations in random domains}.
\newblock {\em SIAM Journal on Scientific Computing}, 28(3):1167--1185, 2007.

\end{thebibliography}

\end{document}